\documentclass[10pt]{article}

\usepackage[utf8]{inputenc}
\usepackage[T1]{fontenc}

\usepackage{amsmath,amssymb,dsfont}
\numberwithin{equation}{section}
\usepackage{microtype}
\usepackage{graphicx,tikz,pgfplots}
\graphicspath{{images/}}
\pgfplotsset{compat=newest}
\usepackage[hyperref,amsmath,thmmarks]{ntheorem}
\usepackage{aliascnt}
\usepackage[a4paper,centering,bindingoffset=0cm,marginpar=2cm,margin=2.5cm]{geometry}
\usepackage[pagestyles]{titlesec}
\usepackage[font=footnotesize,format=plain,labelfont=sc,textfont=sl,width=0.75\textwidth,labelsep=period]{caption}
\usepackage{bm}

\usepackage[backend=biber,maxnames=10,backref=true,hyperref=true,giveninits=true,safeinputenc]{biblatex}

\DefineBibliographyStrings{english}{%
	backrefpage = {cited on page},
	backrefpages = {cited on pages},
}

\title{On convergence rates of adaptive ensemble Kalman inversion for linear ill-posed problems}
\author{Fabian Parzer$^1$\\{\footnotesize\href{mailto:email}{fabian.kai.parzer@univie.ac.at}}
\and Otmar Scherzer$^{1,2,3}$\\{\footnotesize\href{mailto:email}{otmar.scherzer@univie.ac.at}}}
\date{}

\DeclareFieldFormat[report]{title}{``#1''}
\DeclareFieldFormat[book]{title}{``#1''}
\AtEveryBibitem{\clearfield{url}}
\AtEveryBibitem{\clearfield{note}}

\titleformat{\section}[block]{\large\sc\filcenter}{\thesection.}{0.5ex}{}[]
\titleformat{\subsection}[runin]{\bf}{\thesubsection.}{0.5ex}{}[.]

\usepackage[pdftex,colorlinks=true,linkcolor=blue,citecolor=green,urlcolor=blue,bookmarks=true,bookmarksnumbered=true]{hyperref}
\hypersetup
{
    pdfauthor={Fabian Parzer, Otmar Scherzer},
    pdfsubject={Inverse Problems and Regularization},
    pdftitle={On convergence rates of adaptive ensemble Kalman inversion for linear ill-posed problems},
    pdfkeywords={Inverse Problems, Ensemble Kalman inversion}
}

\newpagestyle{headers}
{
	\headrule
	\sethead[\footnotesize\thepage][\footnotesize\sc Fabian Parzer, Otmar Scherzer][]{}{\footnotesize\sc On convergence rates of adaptive ensemble Kalman inversion for linear ill-posed problems
    }{\footnotesize\thepage}
	\setfoot{}{}{}
}
\newtheorem{lemma}{Lemma}[section]

\newaliascnt{proposition}{lemma}
\newtheorem{proposition}[proposition]{Proposition}
\aliascntresetthe{proposition}

\newaliascnt{corollary}{lemma}
\newtheorem{corollary}[corollary]{Corollary}
\aliascntresetthe{corollary}

\newaliascnt{theorem}{lemma}
\newtheorem{theorem}[theorem]{Theorem}
\aliascntresetthe{theorem}

\theorembodyfont{\normalfont}
\newaliascnt{definition}{lemma}
\newtheorem{definition}[definition]{Definition}
\aliascntresetthe{definition}

\newaliascnt{assumption}{lemma}
\newtheorem{assumption}[assumption]{Assumption}
\aliascntresetthe{assumption}

\newaliascnt{example}{lemma}
\newtheorem{example}[example]{Example}
\aliascntresetthe{example}

\newaliascnt{remark}{lemma}
\newtheorem{remark}[remark]{Remark}
\aliascntresetthe{remark}

\newaliascnt{notation}{lemma}

\aliascntresetthe{notation}

\theoremstyle{nonumberplain}
\theoremseparator{:}
\theoremheaderfont{\normalfont\itshape}

\theoremsymbol{\ensuremath{\square}}
\newtheorem{proof}{Proof}

\usepackage{algorithm, algpseudocode}
\usepackage{subcaption}

\usepackage{pbox} 

\newcommand{\N}{\mathds{N}}

\newcommand{\R}{\mathds{R}}

\let\RE\Re
\let\Re=\undefined
\DeclareMathOperator{\Re}{\RE e}
\let\IM\Im
\let\Im=\undefined
\DeclareMathOperator{\Im}{\IM m}
\DeclareMathOperator{\supp}{supp}

\newcommand{\abs}[1]{\left|#1\right|}
\newcommand{\norm}[1]{\left\|#1\right\|}
\newcommand{\set}[1]{\left\{#1\right\}}
\newcommand{\inner}[2]{\left<#1,#2\right>}
\newcommand{\e}{\mathrm e}
\let\ii\i
\renewcommand{\i}{\mathrm i}
\renewcommand{\d}{\,\mathrm d}

\DeclareMathOperator*{\argmin}{argmin}

\newcommand{\Set}[2]{\left\{\,#1:\, #2\,\right\}}
\newcommand{\dom}[1]{ \mathcal{D}(#1) }
\newcommand{\ran}[1]{ \mathcal{R}(#1) }

\newcommand{\rank}{\mathrm{rank}}

\newcommand{\HH}{\mathbb{H}}

\newcommand{\XX}{\mathbb{X}}
\newcommand{\YY}{\mathbb{Y}}

\newcommand{\BL}{\mathcal{L}} 

\newcommand{\Id}{\mathrm{I}}
\newcommand{\Idmat}{\mathbb{I}}

\newcommand{\Exp}[1]{\mathbb{E}\left[ #1 \right] }

\newcommand{\Cov}[1]{ \mathrm{Cov}\left(#1 \right) }

\newcommand{\normal}{\mathcal{N}}

\newcommand{\PP}{\mathbb{P}}

\newcommand{\xtrue}{x_*}
\newcommand{\xo}{x_0}

\newcommand{\xmns}{x^\dagger}
\newcommand{\xtik}{\hat x}
\newcommand{\xtika}{\hat x_\alpha}
\newcommand{\xa}{x_\alpha}
\newcommand{\obs}{\hat y}
\newcommand{\noicov}{R}
\newcommand{\cov}{C_{0}}
\newcommand{\Ij}{\Idmat_J}
\newcommand{\Idx}{\Id_\XX}
\newcommand{\Idy}{\Id_\YY}
\newcommand{\Lp}[2]{\Exp { #1^{#2} }^{1 / {#2} } }

\newcommand{\xinner}[2]{\inner{#1}{#2}_\XX}
\newcommand{\xnorm}[1]{\norm{#1}_\XX}

\newcommand{\Lpxnorm}[2]{ \Lp{\xnorm{#1}}{#2}}

\newcommand{\xopnorm}[1]{\norm{#1}_{\BL(\XX;\XX)}}
\newcommand{\Lpxopnorm}[2]{ \Lp{\xopnorm{#1}}{#2}}
\newcommand{\covnorm}[1]{ \norm{#1}_{\cov} }
\newcommand{\noicovnorm}[1]{ \norm{#1}_{\noicov} }

\newcommand{\Xj}{\hat{X}^{\scriptscriptstyle (J)}}
\newcommand{\Xk}{\hat{X}_k^{\scriptscriptstyle (J)} }

\newcommand{\Xh}{ \hat X^{d,\scriptscriptstyle (J)} }
\newcommand{\Xha}{ \hat X^{d, \scriptscriptstyle (J)}_\alpha }
\newcommand{\Xhvar}[1]{ \hat X^{d, \scriptscriptstyle (#1)} }
\newcommand{\Xad}{\hat X^\mathrm{a}}
\newcommand{\Xaeki}{\hat X^\mathrm{aeki}}
\newcommand{\Xasvd}{\hat x^\mathrm{asvd}}
\newcommand{\Xanys}{\hat X^\mathrm{anys}}
\newcommand{\Xaddet}{\hat x^\mathrm{a}}
\newcommand{\Aj}{ \bm A^{\scriptscriptstyle(J)} }
\newcommand{\Ank}{ {\bm {A}_k^{\scriptscriptstyle (J)}} }
\newcommand{\Ant}{ {\bm{A}^{\scriptscriptstyle (J)}_\mathrm{svd}} }
\newcommand{\Ann}{ {\bm{A}^{\scriptscriptstyle (J)}_\mathrm{nys}} }

\newcommand{\Bj}{ {\bm B^{\scriptscriptstyle(J)}} }
\newcommand{\Bk}{ {\bm{B}_k^{\scriptscriptstyle (J)}} }

\newcommand{\Ck}{\bm C_k^{\scriptscriptstyle (J)}}
\newcommand{\Cj}{\bm C^{\scriptscriptstyle (J)}}
\newcommand{\ens}{\bm U^{\scriptscriptstyle (J)}}
\newcommand{\ensmean}{ \overline{U}^{\scriptscriptstyle (J)} }
\newcommand{\xens}{\bm{X}^{\scriptscriptstyle (J)}}
\newcommand{\xensmean}{ \overline{X}^{\scriptscriptstyle (J)} }
\newcommand{\Encov}{\mathcal{C}}
\newcommand{\Anom} { \mathcal{A} }

\newcommand{\gain}{\mathcal K}
\newcommand{\gainfun}{K}
\newcommand{\wens}{ {\bm{W}^{\scriptscriptstyle (J)}} }
\newcommand{\Qj}{ {\bm{Q}^{\scriptscriptstyle (J)}} }
\newcommand{\Rj}{ \bm R^{\scriptscriptstyle (J)} }

\newcommand{\akdo}{ \alpha_{K_\delta(\omega)} }

\newcommand{\kdo}{ K_\delta(\omega) }
\newcommand{\wtik}{ \hat w }

\newcommand{\xoball}{{\overline{B}_r(x_0)}}

\newcommand{\projball}{P_r}
\newcommand{\projectedeki}{P_r \left( \Xhvar{J_k}_{\alpha_k}(\omega) \right)}
\newcommand{\goodset}{E_\mathrm{good}(\delta)}
\newcommand{\goodsetk}{E_\mathrm{good}^k}

\begin{document}

\maketitle
\thispagestyle{empty}
\begin{center}
\hspace*{5em}
\parbox[t]{12em}{\footnotesize
\hspace*{-1ex}$^1$Faculty of Mathematics\\
University of Vienna\\
Oskar-Morgenstern-Platz 1\\
A-1090 Vienna, Austria}
\hfil
\parbox[t]{17em}{\footnotesize
\hspace*{-1ex}$^2$Johann Radon Institute for Computational\\
\hspace*{1em}and Applied Mathematics (RICAM)\\
Altenbergerstraße 69\\
A-4040 Linz, Austria}
\end{center}
\begin{center}
\parbox[t]{19em}{\footnotesize
\hspace*{-1ex}$^3$Christian Doppler Laboratory\\
for Mathematical Modeling and Simulation\\
of Next Generations of Ultrasound Devices (MaMSi)\\
Oskar-Morgenstern-Platz 1\\
A-1090 Vienna, Austria}
\end{center}

\begin{abstract}
In this paper we discuss a deterministic form of ensemble Kalman inversion as a regularization method for linear inverse problems. By interpreting ensemble Kalman inversion as a low-rank approximation of Tikhonov regularization, we are able to introduce a new sampling scheme based on the Nyström method that improves practical performance. Furthermore, we formulate an adaptive version of ensemble Kalman inversion where the sample size is coupled with the regularization parameter. We prove that the proposed scheme yields an order optimal regularization method under standard assumptions if the discrepancy principle is used as a stopping criterion. The paper concludes with a numerical comparison of the discussed methods for an inverse problem of the Radon transform.
\end{abstract}

\section{Introduction}\label{sec:intro}

In recent years, \emph{ensemble Kalman inversion} (EKI) has become a popular tool for solving inverse problems \cite{IglLawStu13}. EKI has advantages against other 
iterative methods in situations where the evaluation of the forward operator is costly, and information about its adjoint or its derivative is unavailable.

While there are some recent results on the convergence of EKI as an optimization method \cite{SchiStu17b, SchiStu17, ChaTon21, Wei22_report}, the regularization theory of EKI is still incomplete.
In this paper, we provide an analysis of a deterministic form of EKI as a regularization method for solving \emph{linear inverse problems}. That is, we consider the problem of determining a solution $\xtrue$ of the linear operator equation
\begin{align}
y = L \xtrue, \label{eq:inverseProblem}
\end{align}
where $L:\XX \to \YY$ is a bounded linear operator between Hilbert spaces. We do not assume that we have access to $y$, but only to a noisy measurement
\begin{align}
\obs = y + \xi, \label{eq:noisy_data}
\end{align}
where $\xi$ is noise. 

Such an analysis is important for three reasons: First, it allows a theoretical comparison of EKI with established iterative regularization methods for inverse problems, such as the iteratively regularized Gauss-Newton \cite{Bak92} or the iteratively regularized Landweber \cite{Sch98} iteration. Secondly, it allows the transfer of knowledge between functional-analytic regularization theory, in particular the study of finite-dimensional approximation of Tikhonov regularization \cite{Gro84, NeuSch90}, and the emerging literature on ensemble methods for the solution of inverse problems (see for example \cite{Igl14} or \cite{RaaStoEve19}). Finally, this analysis can potentially serve as the basis for a generalized analysis of EKI for nonlinear inverse problems, making use of the deterministic convergence analysis of iterative regularization methods in Hilbert space (see \cite{Bak92,HanNeuSch95,KalNeuSch08}).

It was already noted in \cite{IglLawStu13} that in the case of a linear operator equation the first iteration of EKI converges to the Tikhonov regularized solution as the sample size approaches infinity. It can be shown that -- at least for the deterministic version considered in this paper -- this also holds true for all subsequent iterates, where each iterate is associated with a different choice of regularization parameter. Thus, in the linear case, EKI can be completely characterized as a stochastic low-rank approximation of Tikhonov regularization. As a consequence we can prove that under appropriate source conditions and by adapting the sample size to the regularization parameter (this method is then called adaptive EKI), we get optimal convergence rates for EKI in the sense formulated for instance in \cite{EngHanNeu96}. Moreover, we show that the efficiency of EKI can be increased by the use of more sophisticated low-rank approximation schemes, such as the Nyström method (see e.g. \cite{GitMah16}).

The paper is organized as follows:

\begin{itemize}
\item We continue this section by recalling some required notation and functional-analytic prerequisites (\autoref{sec:notation_terminology}) and providing an appropriate definition 
of the deterministic form of EKI that is considered for the rest of this paper (see \autoref{sec:eki}).
\item In \autoref{sec:direct}, we discuss deterministic EKI as an approximation to Tikhonov regularization. In particular, we derive error estimates in dependence of the regularization parameter which build the foundation for the subsequent formulation of an adaptive version. In \autoref{subsec:sampling} we review some results and methods for the low-rank approximation of operators, in particular the Nyström method. We show how these methods naturally lead to new versions of EKI. 
\item In \autoref{sec:rates} we propose an adaptive variant of EKI. The algorithm is described in \autoref{sec:adaptive_description} and analyzed as an iterative regularization method in \autoref{sec:adaptive_analysis}, where we describe conditions under which we can prove optimal convergence rates in the zero-noise limit. This constitutes our main result. Further remarks comparing the proposed scheme with similar methods from the existing literature are given in \autoref{sec:general_remarks}.
\item We conclude our paper in \autoref{sec:numerics} with numerical experiments in the context of computerized tomography. These experiments demonstrate some advantages and shortcomings of EKI for linear inverse problems. In particular, they show that the Nyström EKI method leads to considerable improvements 
in terms of numerical performance in comparsion to existing sampling methods.
\item The appendix reviews some prerequisites from probability theory, and discusses how our exposition relates to alternative formulations of EKI that have been studied elsewhere.
\end{itemize}

\subsection{Notation and terminology}\label{sec:notation_terminology}
We summarize basic notation first:
\begin{enumerate}
\item $\XX$ and $\YY$ denote real separable Hilbert spaces. 
\item $\BL(\XX;\YY)$ denotes the space of bounded linear operators from $\XX$ to $\YY$.  
\item If $L:\XX \to \YY$ is a linear operator, we let $\dom L \subset \XX$ denote its domain and $\ran L \subset \YY$ denote its range. 
\item We call $P \in \BL(\XX;\XX)$ positive if $\xinner{Px}{x} \geq 0$ for all $x \in \XX$. 
\item For a positive and self-adjoint operator $P \in \BL(\XX;\XX)$, we define the $P$-weighted norm
\begin{align*}
\norm{x}_P = \begin{cases}
\xnorm{P^{-1/2} x}, & \text{if } x \in \ran{P^{1/2}}, \\
\infty, & \text{else},
\end{cases}
\end{align*}
where the operator $P^{-1/2}$ is defined as the pseudoinverse of $P^{1/2}$, which in turn can be defined via spectral theory, see for example \cite[chapter 2.3]{EngHanNeu96}.

\item \emph{Trace class:}
We say that an operator $P \in \BL(\XX;\XX)$ is in the \emph{trace class} if for any orthonormal basis $(e_n)$ of $\XX$ we have
\begin{align*}
\sum_n \abs{\xinner{P e_n}{e_n}} < \infty.
\end{align*}
\item $(\Omega, \mathcal F, \PP)$ denotes a probability space.
\end{enumerate}

\subsection{Ensemble Kalman inversion for linear inverse problems}\label{sec:eki}

Next, we present a particular form of the EKI iteration associated to problem \eqref{eq:inverseProblem}. The original form of EKI \cite{IglLawStu13}, which we refer to as \emph{stochastic EKI}, evolves a random ensemble through an iteration where additional noise is added in each step. In the last few years, multiple variants of EKI have been developed that incorporate adaptable stepsizes \cite{KovStu19, ChaTon21} or additional regularization \cite{ChaStuTon20}. In particular, one can also formulate a deterministic version that circumvents the addition of noise by directly transforming the ensemble mean and covariance. Such a version of EKI has for example been considered in \cite{ChaTon21}. In accordance with the literature on ensemble Kalman filtering, we will refer to this as deterministic EKI \cite{TipAndBisHamWhi03, HouZha16}. A more detailed discussion of its relation to the stochastic form of EKI can be found in \autoref{sec:eki_with_perturbations}.

The EKI iteration involves two linear operators $\cov:\XX \to \XX$ and $\noicov: \YY \to \YY$ that characterize regularity assumptions on the solution $\xtrue$ and the noise $\xi$.
They have to be provided by the practitioner to represent prior information on the problem. In the rest of this article, we will assume that they satisfy the following conditions:

\begin{assumption}\label{assumption}
Let $\cov \in \BL(\XX;\XX)$ and $\noicov \in \BL(\YY;\YY)$ be injective, positive and self-adjoint linear operators such that
\begin{enumerate}
\item $\cov$ is compact,
\item $\ran{L} \subset \dom{ \noicov^{-1/2} }$, and there exists a constant $c_{RL}\in \R$ such that
\begin{equation}\label{eq:rl_condition}
\norm{R^{-1/2} L}_{\BL(\XX;\YY)} \leq c_{RL}.
\end{equation}
\end{enumerate} Moreover, we assume that the noisy data $\obs$ defined in \autoref{eq:noisy_data} satisfies $\obs \in \ran{\noicov^{1/2}}$.
\end{assumption}

As the next proposition shows, the subspace $\dom{ \cov^{-1/2} } \subset \XX$ together with the norm $\covnorm{\cdot}$ yields a Hilbert space. This space will play an important role for our analysis in \autoref{sec:rates}.

\begin{proposition}\label{c_space}
Let $\cov \in \BL(\XX;\XX)$ be an injective, positive and self-adjoint bounded linear operator. 
Let
\begin{align*}
\inner{x}{y}_{\cov} := \xinner{\cov^{-1/2}x}{\cov^{-1/2}y} \quad \text{ for all } x,y \in \dom{\cov^{-1/2}}.
\end{align*}
Then $\dom{\cov^{-1/2}}$ equipped with the inner product $\inner{\cdot}{\cdot}_{\cov}$ defines a Hilbert space, denoted by $\XX_{\cov}$. Moreover
\begin{align}
\xnorm{x} \leq \xopnorm{\cov^{1/2}} \covnorm{x} \quad \text{ for all } x \in \XX_{\cov}. \label{eq:c_norm_stronger}
\end{align}
\end{proposition}

\begin{proof}
The bilinear form $\inner{\cdot}{\cdot}_{\cov}$ is well-defined on $\dom{\cov^{-1/2}} = \ran{\cov^{1/2}}$ because $\cov$ is injective. Furthermore, this bilinear form is symmetric and positive semidefinite because $\cov^{-1/2}$ is self-adjoint and positive. The definiteness follows from the injectivity of $\cov^{-1/2}$.
\autoref{eq:c_norm_stronger} follows from the boundedness of $\cov$, since we have
\begin{align*}
\xnorm{x} \leq \xopnorm{\cov^{1/2}} \xnorm{\cov^{-1/2} x} = \xopnorm{\cov^{1/2}} \covnorm{x} \quad \text{for all } x \in \dom{\cov^{-1/2}}.
\end{align*}
Finally, the completeness of $\dom{\cov^{-1/2}}$ with respect to $\norm{\cdot}_{\cov}$ is a direct consequence of the completeness of $\XX$.
\end{proof}

\begin{remark}
At this point, we want to stress that the operator $\noicov$ does \emph{not} correspond to the assumption that $\xi = \obs - y$ is a Gaussian random element of $\YY$ with covariance $\noicov$. In fact, in the case where $\YY$ is infinite-dimensional, one can show that $\noicovnorm{\xi} = \infty$ with probability 1 (see \cite[theorem 2.4.7]{Bog98}). The proper interpretation of $\noicov$ is that it determines a subspace $\YY_R \subset \YY$ in which $\xi$ is assumed to lie (see \autoref{c_space}).
\end{remark}

Before we continue with the description of the deterministic EKI iteration, we present an illustrative example for a choice of the operators $C_0$ and $R$ that is often used in practice.

\begin{example} \label{ex:c0}
If we let $\XX = L^2(D)$ and $\YY = L^2(E)$, where $D \subset \R^{d_1}$ and $E \subset \R^{d_2}$ are bounded domains with piecewise smooth boundaries. Consider the choice $\cov = (\Idx - \Delta)^{-1}$ and $R = \Idy - \Delta$. Then the operator $\cov$ is compact. Here, $(\Idx - \Delta)^{-1}$ is the operator which maps a given function $\rho \in \XX$ onto the weak solution of the equation
\begin{equation*}
\begin{aligned} 
(\Idx - \Delta)u &= \rho \text{ in } D,\\
\frac{\partial u}{\partial n} &= 0 \text{ on } \partial D,
\end{aligned}
\end{equation*}
The range of $\cov^{1/2}$ is $H^1(D)$, i.e. the Sobolev space of first order. It is easy to see that $\cov^{-1}$ is positive and self-adjoint, and thus so is $\cov$.
We also have
\begin{equation*}
\norm{u}_{\cov}^2 = \int_D \left((I - \Delta)^{1/2} u\right)^2 d \vec{x} = 
\int_D u \left((\Idx - \Delta)u \right) d \vec{x} = \int_D u^2 + \abs{\nabla u}^2 d \vec{x} = \norm{u}_{H^1(D)}^2.
\end{equation*}
Similarly
\begin{equation*}
\noicovnorm{v}^2 = \norm{v}_{H^{-1}(E)}^2,
\end{equation*}
where $H^{-1}(E)$ denotes the dual space of $H^1(E)$.
\end{example}

The fundamental difference of ensemble methods to existing regularization methods is the use of a stochastic low-rank approximation of $\cov$, which reduces the effective dimension of the parameter space $\XX$. The next definition gives this notion a precise meaning.

\begin{definition}[Low-rank approximation] \label{de:lra}
Let $\cov \in \BL(\XX;\XX)$ be a self-adjoint, positive and compact linear operator and let $\gamma > 0$.
\begin{enumerate}
\item Let $(\Aj)_{J=1}^\infty$ be a family of bounded linear operators with $\Aj \in \BL(\R^J;\XX)$ for all $J \in \N$. We say that it generates a \emph{deterministic low-rank approximation} of $C_0$, \emph{of order $\gamma$}, if there exists a constant $\nu$ such that
\begin{align*}
\xopnorm{\Aj {\Aj}^* - \cov} \leq \nu J^{-\gamma} \qquad \text{ for all } J \in \N.
\end{align*}
\item Let $p \in [1,\infty)$ and $(\Aj)_{J=1}^\infty$ be a family of random bounded linear operators (see \autoref{sec:appendix}) with $\Aj(\omega) \in \BL(\R^J;\XX)$ for all $\omega \in \Omega$ and $J \in \N$. We say that it generates a \emph{stochastic low-rank approximation} of $\cov$, \emph{of $p$-order $\gamma$}, if there exists a constant $\nu_p$ such that
\begin{align*}
\Lpxopnorm{\Aj {\Aj}^* - \cov}{p} \leq \nu_p J^{-\gamma} \qquad \text{ for all } J \in \N.
\end{align*}

\end{enumerate}
\end{definition}

Under \autoref{assumption}, the following algorithm is well-defined, for all $k \in \N$.

\begin{definition}[Deterministic EKI] \label{de:EKI} Let $(\Aj)_{J=1}^\infty$ generate a low-rank approximation of $\cov$, and let $\obs \in \YY$, $J \in \N$, and an initial guess $\xo \in \XX$ be given. 
\begin{itemize}
\item \emph{Initialization:} Set $\Xj_0 := x_0$ and $\Aj_0 := \Aj$.
\item \emph{Iteration ($k \to k+1$):} Let $\Bk = \noicov^{-1/2} L \Ank : \R^J \to \YY$, and set
\begin{align}
& \Xj_{k+1} = \Xk + \Ank \left( \Bk^* \Bk + \Ij \right)^{-1} \Bk^* \noicov^{-1/2} (\obs - L \Xk),\label{eq:eki1}\\
\text{and} \quad & \Aj_{k+1} = \Ank \left( \Bk^* \Bk + \Ij \right)^{-1/2},\label{eq:eki2}
\end{align}
where $\Ij \in \R^{J\times J}$ denotes the identity matrix and $\Bk^*:\YY \to \R^J$ denotes the adjoint of $\Bk$.
\end{itemize}
\end{definition}

Note that the adjective "deterministic" in \autoref{de:EKI} refers only to the update formula, which -- in contrast to the original, stochastic EKI iteration (see \autoref{de:EKI_perturbed}) -- does not introduce additional noise. Even if a stochastic low-rank approximation is used in \autoref{de:EKI}, we will refer to the resulting method as deterministic EKI. In this case, the algorithm is defined pointwise, for every $\omega \in \Omega$. That is, the quantities $\Xk$, $\Ank$ and $\Bk$ all depend on $\omega$. 
For the rest of this paper, we will suppress this dependence. 
This allows us to treat both deterministic and stochastic low-rank approximations at once.

\begin{remark}\label{rem:covariance_form}
We have introduced the EKI update equations (\autoref{eq:eki1}-\autoref{eq:eki2}) in the so-called \emph{square-root form}. It is equivalent (see e.g. \cite{TipAndBisHamWhi03}) to the so-called \emph{covariance form} which is more widespread in the literature on the Kalman filter and given by
\begin{equation}\label{eq:covariance_form}
\begin{aligned}
& \Xj_{k+1} = \Xk + \Ck L^* \left(L \Ck L^* + \noicov \right)^{-1} (\obs - L \Xk),\\
& \Cj_{k+1} = \Ck - \Ck L^* \left(L \Ck L^* + \noicov \right)^{-1}L \Ck.
\end{aligned}
\end{equation}
The operator $\Ck$ is related to $\Ank$ from \autoref{de:EKI} via the identity $\Ck = \Ank \Ank^*$, which holds for all $k \in \N$. The computational difference between these two formulations is that the square-root form requires the inversion of an operator on $\R^J$, while the covariance form requires inversion of an operator on $\YY$.
\end{remark}

The existing literature on EKI focuses mostly on the case where the low-rank approximation $(\Aj)_{J=1}^\infty$ is generated by the so-called \emph{anomaly operator} of an ensemble $\ens$ of random elements -- thus the name ``\emph{ensemble} Kalman inversion''. That is, one uses $\Aj = \Anom(\ens)$, where $\Anom(\ens)$ is defined as follows:

\begin{definition}[Ensemble anomaly]\label{de:anomaly}
A $J$-tuple $\ens = (U_1,\ldots,U_J)$ of random elements $U_1^{(J)},\ldots,U_J^{(J)}$ of $\XX$ is called a \emph{random ensemble}. We call the random element 
\begin{equation} \label{eq:mean}
  \ensmean := \frac{1}{\sqrt{J}} \sum_{j=1}^J U_j
\end{equation}
the \emph{ensemble mean}. Furthermore, we call the random continuous linear operator from $\R^J$ to $\XX$ (see \autoref{sec:appendix}) defined by 
\begin{equation} \label{eq:anamoly}
\Anom(\ens)v := \frac{1}{\sqrt J} \sum_{j=1}^J v_j (U_j - \ensmean) \qquad \text{for all } v \in \R^J,
\end{equation}
the \emph{ensemble anomaly}.
\end{definition}

We will see in \autoref{subsec:sampling} that $(\Anom(\ens))_{J=1}^\infty$ generates a stochastic low-rank approximation of $\cov$ if $U_1,\ldots,U_J$ are independent Gaussian random elements with $\Cov{U_j^{(J)}} = \cov$, for all $j=1,\ldots,J$. However, the more general \autoref{de:EKI} allows us to consider other forms of low-rank approximations, in particular also deterministic ones (see \autoref{subsec:sampling}).

\begin{remark}
The update \autoref{eq:eki1} can also be expressed as the solution to a minimization problem, since for all $k \in \N$, $\Xj_{k+1}$ is the minimizer 
of the functional 
\begin{equation} \label{eq:galerkin}
x \in \mathcal{D}(\Ank^*) \longmapsto \norm{L x - \obs}^2_R + \norm{x-\Xk}^2_{\Ank(\Ank)^*},
\end{equation}
which is well-defined due to \autoref{assumption}.
\end{remark}

\section{EKI as approximate Tikhonov regularization}

\subsection{Direct EKI} \label{sec:direct}

Ensemble Kalman methods originated in data assimilation \cite{Eve94} and are traditionally applied to state estimation in dynamical systems \cite{NakPot15} \cite{ReiCot15}. Following this logic, EKI, which has been developed for the treatment of inverse problems, is often analyzed as a nonstationary regularization method with multiple steps, where the iteration number $k$ controls the amount of regularization. For the deterministic version of EKI given by \autoref{eq:eki1} and \autoref{eq:eki2}, one can actually show that multiple iterations with initial covariance operator $\cov$ are equivalent to a single iteration with covariance operator $\tilde \cov = \frac{1}{k} \cov$. This result can be seen as direct consequence of the classical equivalence of the Kalman filter to four-dimensional variational data assimilation (4D-VAR) \cite{RauTunStr65}.

\begin{theorem}\label{generatorThm} Let $\Bj := \noicov^{-1/2} L \Aj$, and let $(X_k)_{k=1}^\infty$ denote the EKI iteration as defined in \autoref{de:EKI}. Then, the following representation holds
\begin{align}
& \Xk = \xo + \Aj \left( \Bj^* \Bj + k^{-1} \Ij \right)^{-1} \Bj^* \noicov^{-1/2} (\obs - L \xo) \qquad \text{for all } k \in \N. \label{eq:entik1}
\end{align}
\end{theorem}

\begin{proof}

Follows from \cite[theorem 5.4.7]{NakPot15} by setting $M=\Idx$, $H_\xi = L$ and $f^{(\xi)} = \obs$ for $\xi=1,\ldots, k$.

\end{proof}

A first consequence of \autoref{generatorThm} is that it allows us to embed EKI into a parameter-dependent family of operators, which we will call direct EKI:

\begin{definition}[Direct EKI]\label{de:direct}
Suppose that \autoref{assumption} holds, and let $\alpha > 0$. Then, we define the \emph{direct EKI} in the following way
\begin{equation} \label{eq:eki_alpha1}
\begin{aligned}
\Xha & := \xo + \gainfun_\alpha(\Aj) (\obs - L \xo),\\
\text{where} \quad \gainfun_\alpha(\Aj) & := \Aj \left({\Aj}^* L^* \noicov^{-1} L \Aj + \alpha\Ij \right)^{-1} {\Aj}^* L^*R^{-1}.
\end{aligned}
\end{equation}
\end{definition}

According to \autoref{eq:entik1}, we have 
\begin{equation} \label{eq:emb_versus}
\Xh_{1/k} = \Xk.
\end{equation}
That is, the $k$-th iterate of deterministic EKI is equivalent to direct EKI with the choice $\alpha = 1/k$.


Next, we derive error estimates between direct EKI and Tikhonov regularization in terms of the sample size $J$ and the regularization parameter $\alpha$. To this end, let us recall the notion of the Tikhonov-regularized solution of \autoref{eq:inverseProblem}.

\begin{definition}[Tikhonov regularization]\label{de:tikhonov}
Let \autoref{assumption} hold. Then 
the unique minimizer of 
\begin{equation}\label{eq:tikhonov_optimality}
x \in \XX \longmapsto \norm{\obs - Lx}_R^2 + \alpha \norm{x-\xo}_{\cov}^2 
\end{equation}
is called the \emph{Tikhonov regularized solution of \autoref{eq:inverseProblem} according to the data $\obs$ and the regularization parameter $\alpha$}. It is denoted 
with $\xtika$ and explicitly represented by
\begin{equation} \label{eq:tikhonov_formula}
\begin{aligned}
\xtika & := \xo + \gain_\alpha (\obs - L \xo), \\ 
\text{where} \quad \gain_\alpha & := \cov^{1/2} \left(\cov^{1/2} L^* \noicov^{-1} L \cov + \alpha \Idx \right)^{-1} \cov^{1/2} L^* R^{-1},
\end{aligned}
\end{equation}
and where $\Idx:\XX \to \XX$ is the identity operator.
\end{definition}

\begin{remark}
We emphasize the notational difference between \autoref{eq:eki_alpha1} and \autoref{eq:tikhonov_formula} that $\Idx$ denotes the identity operator on 
$\XX$ while $\Ij \in \R^{J \times J}$ denotes the identity matrix for $\R^J$.
\end{remark}

\begin{example}\label{ex:r}
Consider again \autoref{ex:c0}. In that case \autoref{eq:tikhonov_optimality} becomes 
\begin{equation*}
x \in \XX \longmapsto \norm{\obs - Lx}_{H^{-1}(E)}^2 + \alpha \norm{x-\xo}_{H^1(D)}^2.
\end{equation*}
\end{example}

If we compare \autoref{eq:eki_alpha1} and \autoref{eq:tikhonov_formula}, we observe that the main difference between Tikhonov regularization and direct EKI is the replacement of the operator $\gain_\alpha$ (Tikhonov) by a low-rank approximation $\gainfun_\alpha(\Aj)$ (direct EKI). In the following the 
difference between the random element $\Xha$ and the Tikhonov regularized solution $\xtika$ is estimated.

\begin{lemma}[Tikhonov versus direct EKI] \label{ekiEstimate}
Let $\alpha > 0$, $p \in [1,\infty)$, and suppose that \\ \autoref{assumption} holds. Then there exists a constant $c$, independent of $J$, such that
\begin{equation}\label{eq:ekiErrorEstimate}
\xnorm{\Xha - \xtika} \leq c \cdot \phi(\alpha)  \xopnorm{\Aj {\Aj}^* - \cov} \qquad \text{for all } J \in \N,
\end{equation}
where $\phi(\alpha) := \max(\alpha^{-1}, \alpha^{-2})$.
\end{lemma}

\begin{proof}
By \autoref{eq:tikhonov_formula} and \autoref{eq:eki_alpha1} we have
\begin{equation} \label{eq:eki_versus_tikhonov}
\Xha - \xtika = (\gainfun_\alpha(\Aj) R^{1/2} - \gain_\alpha R^{1/2}) R^{-1/2}(\obs - L \xo). 
\end{equation}
Using spectral theory, one can show
\begin{align}
\xopnorm{(P + \alpha \Idx)^{-1}} & \leq \alpha^{-1},\label{eq:spectral_estimate1} \\
\text{and} \qquad \xopnorm{(P + \alpha \Idx)^{-1} - (Q + \alpha \Idx)^{-1}} & \leq \alpha^{-2} \xopnorm{P-Q}, \label{eq:spectral_estimate2}
\end{align}
for all positive and self-adjoint bounded linear operators $P$ and $Q$ (see \cite[section 2.3]{EngHanNeu96}). Furthermore, recall that every linear operator $A$ satisfies the identity $(A^*A + \alpha \Idx)^{-1}A^* = A^* (AA^* + \alpha \Idy)^{-1}$ if one of these expressions is well-defined. With this, one can show that
\begin{align*}
\gainfun_\alpha(\Aj) &= \Aj {\Aj}^* L^* R^{-1/2} \left(\Bj \Bj^* + \alpha \Idy \right)^{-1} R^{-1/2},\\
\text{and} \quad \gain_\alpha & = \cov^* L^* R^{-1/2} \left(B B^* + \alpha \Idy \right)^{-1} R^{-1/2},
\end{align*}
where we used the notation $\Bj := R^{-1/2} L \Aj$ and $B := R^{-1/2} L C_0^{1/2}$ for brevity. These identities imply
\begin{align*}
\gainfun_\alpha(\Aj) - \gain_\alpha & = (\Aj {\Aj}^* - \cov) R^{-1/2} L \left(\Bj \Bj^* + \alpha \Idy \right)^{-1} R^{-1/2} \\
& \quad + \cov L^* R^{-1/2} \left[ \left(\Bj \Bj^* + \alpha \Idy \right)^{-1} - \left(B B^* + \alpha \Idy \right)^{-1}\right]R^{-1/2}.
\end{align*}
Taking norms and using \autoref{eq:spectral_estimate1}, \autoref{eq:spectral_estimate2}, and \autoref{eq:rl_condition}, we then obtain
\begin{align*}
\norm{\gainfun_\alpha(\Aj)R^{1/2} - \gain_\alpha R^{1/2}}_{\BL(\YY;\XX)} \leq c_{RL} (1 + c_{RL}^2 \xopnorm{C_0}\alpha^{-1}) \alpha^{-1} \xopnorm{\Aj {\Aj}^* - \cov}.
\end{align*}
Taking norms in \autoref{eq:eki_versus_tikhonov} and inserting this estimate proves the assertion.
\end{proof}

This lemma shows that the difference between Tikhonov regularization and direct EKI can be bounded in terms of the difference between the operators $\Aj {\Aj}^*$ and $\cov$. If this difference decreases with a certain rate with respect to $J$, then direct EKI converges to Tikhonov regularization with the same rate. 

\begin{proposition}[Convergence of EKI to Tikhonov]\label{large_ensemble_convergence} Let \autoref{assumption} hold. 
\begin{enumerate}
\item If $(\Aj)_{J=1}^\infty$ generates a deterministic low-rank approximation of $\cov$ of order $\gamma$, then there exists a constant $\kappa$ such that
\begin{align*}
\xnorm{\Xha - \xtika} \leq \kappa \phi(\alpha) J^{-\gamma} \qquad \text{for all } \alpha > 0 \text{ and all } J \in \N.
\end{align*}
\item  Let $p \in [1,\infty)$. If $(\Aj)_{J=1}^\infty$ generates a stochastic low-rank approximation of $\cov$ of $p$-order $\gamma$, then there exists a constant $\kappa_p$ such that
\begin{align*}
\Lpxnorm{\Xha - \xtika}{p} \leq \kappa_p \phi(\alpha) J^{-\gamma} \qquad \text{for all } \alpha > 0 \text{ and all } J \in \N.
\end{align*}
\end{enumerate}
\end{proposition}

\begin{proof}
Follows directly from \autoref{ekiEstimate} and \autoref{de:lra} with $\kappa=\nu \cdot c$ and $\kappa_p = \nu_p \cdot c$.
\end{proof}

\begin{remark}
Alternatively to the above derivation, the convergence of deterministic EKI to Tikhonov regularization can be seen as special case of the convergence of the ensemble square-root filter to the Kalman filter, see for example \cite{KwiMan15} or \cite[section 5.4]{NakPot15}. However, the alternative results presented here are better suited to investigate convergence rates of EKI as a regularization method (see \autoref{sec:rates}), since they explicitly describe the dependence of the approximation error on the regularization parameter $\alpha$. We also note that different types of finite-dimensional approximations of Tikhonov approximation have been studied elsewhere, for example in \cite{Gro84, NeuSch90}.
\end{remark}

\subsection{Optimal low-rank approximations for EKI}\label{subsec:sampling}

\autoref{large_ensemble_convergence} shows that direct EKI, and thus also EKI, converges to Tikhonov regularization with rate equal to the order of the employed low-rank approximation. In general, convergent low-rank approximations only exist if the eigenvalues of $\cov$ satisfy a decay condition.

\begin{assumption}[Decreasing eigenvalues of $\cov$] \label{eigen_decay}
Let $\cov$ satisfy \autoref{assumption}, and let $(\lambda_n)$ denote its eigenvalues in decreasing order. We assume that there exists a constant $\eta > 0$ such that
\begin{align*}
\lambda_n = O(n^{-\eta}).
\end{align*}
\end{assumption}

\begin{remark}
In this paper, we always assume that all eigenvalues are repeated according to their multiplicities.
\end{remark}

\begin{example} Consider \autoref{ex:c0}. In this case, \autoref{eigen_decay} is satisfied with $\eta = 2/d$ \cite{Kro92}.
\end{example}

Under \autoref{eigen_decay}, the Schmidt-Eckhardt-Young-Mirsky theorem \cite{Schm07} \cite{EckYou36} \cite{Mir60} states that the best possible order of any low-rank approximation of $\cov$ is $\eta$, and it is achieved by the truncated singular value decomposition.

\begin{theorem}[Schmidt-Eckhardt-Young-Mirsky]\label{schmidt}
Let $\cov \in \BL(\XX;\XX)$ be positive, self-adjoint and compact, and let $(\lambda_n)$ denote its eigenvalues in decreasing order. Let $\Ant \Ant^*$ denote the $J$-truncated singular value decomposition of $\cov$. Then
\begin{align*}
\xopnorm{\Ant \Ant^* - \cov} = \lambda_{J+1} = \inf \Set{\xopnorm{\bm{P} - \cov}}{\bm P \in \BL(\XX;\XX),~ \rank (\bm P) \leq J}.
\end{align*}
\end{theorem}

\begin{remark}
Note that the optimal possible order for a low-rank approximation does not directly depend on the dimension of the underlying spaces $\XX$ and $\YY$, only on the decay of the eigenvalues of $\cov$. This means that we obtain dimension-independent convergence rates as long as the eigenvalues of $\cov$ decay sufficiently fast.
\end{remark}

Existing formulations of EKI generate a stochastic low-rank approximation of $\cov$ from the ensemble anomaly $\Anom(\ens)$ (see \autoref{de:anomaly}) of a randomly generated ensemble $\ens$. For this type of approximation, we have the following result.

\begin{theorem}[Low rank approximation of $\cov$]\label{sample_covariance}
Assume that $\cov$ is in the trace-class and let $p \in [1,\infty)$ be fixed. Moreover, for every $J \in \N$, let $\ens = [U_1,\ldots,U_J] \in \XX^J$ be an ensemble of independent Gaussian random elements with $\Cov{U_j} = \cov$, for all $j \in \{ 1,\ldots, J \}$, and let $\Anom(\ens)$ be as in \autoref{de:anomaly}. Then
$(\Anom(\ens))_{J=1}^\infty$ generates a stochastic low-rank approximation of $\cov$, 
of $p$-order $1/2$, meaning that there exists a constant $\nu_p$ such that
\begin{align*}
\Lpxopnorm{\Anom(\ens)\Anom(\ens)^* - \cov}{p} \leq \nu_p J^{-1/2} \qquad \text{for all } J \in \N.
\end{align*}
In particular, for $p=1$, there exists a constant $c > 0$ such that
\begin{align}
c J^{-1/2} \leq \Exp{\xopnorm{\Anom(\ens) \Anom(\ens)^* - \cov}} \qquad \text{for all } J \in \N. \label{eq:lower_bound}
\end{align}
\end{theorem}

\begin{proof}
See \cite{KolLou17}.
\end{proof}

\begin{example}\label{ex:covariance_kernel}
We want to give some examples of trace-class operators on $L^2(\mu)$, where $\mu$ is a Radon measure on a domain $U \subset \R^d$ with $\supp \mu = U$. Then, Mercer's theorem (see e.g.  \cite[Theorem 5.6.9]{Dav07}) characterizes a large class of trace-class operators: An operator $P: L^2(\mu) \to L^2(\mu)$ is in the trace-class if it can be represented by an integrable continuous positive-definite kernel, i.e.
\begin{align*}
P f(x) = \int_U k(x,y)f(y) \d \mu(y).
\end{align*}
\end{example}

The following result on trace-class operators allows us to directly compare the order of $(\Anom(\ens))_{J=1}^\infty$ to the theoretical optimum defined in \autoref{schmidt}.

\begin{proposition}\label{operator_stuff}
Let $\cov$ be a positive and self-adjoint trace-class operator with eigenvalues $(\lambda_n)$. Then
\begin{align*}
\lambda_n = O(n^{-1}).
\end{align*}
\end{proposition}

\begin{proof}
It follows from the assumptions on $\cov$ that
\begin{align*}
\sum_{n=1}^\infty \lambda_n < \infty,
\end{align*}
(see e.g. \cite[lemma 5.6.2]{Dav07}) which implies $\lambda_n = O(n^{-1})$.
\end{proof}

Therefore, if $\cov$ is in the trace-class, then according to \autoref{schmidt} the optimal low-rank approximation of $\cov$ is given by $\Ant \Ant^*$ and is at least of order 1. However, since the low-rank approximation generated by $(\Anom(\ens))_{J=1}^\infty$ satisfies the lower bound \autoref{eq:lower_bound}, the ensemble-based low-rank approximation, while cheaper, is not of optimal order.

This leads to the question whether there exist low-rank approximations of $\cov$ that are of optimal order but do not require knowledge of the singular value decomposition of $\cov$. 
The answer to this question is yes. There exist stochastic low-rank approximations that are of optimal order and only require $O(J)$ evaluations of $\cov$ \cite{HalMarTro11}. An example of such a scheme is the \emph{Nyström method} \cite{Nys30} \cite{GitMah16} \cite{Nak20_report}. We will consider a special case given by algorithm 1.

\begin{algorithm}[H]
\caption{\texttt{Nyström method (with projection-based sketches)}}
Given a positive and self-adjoint operator $\cov$ and a target rank $J \in \N$.
\begin{algorithmic}[1]
\State Generate iid random samples $W_1,\ldots,W_J \sim \normal(0, \Idx)$ and assemble them in $\wens = [W_1,\ldots,W_J]$;
\State Compute $\ens = \cov \wens$;
\State Compute the reduced $QR$-decomposition $\ens = \Qj \Rj$, where $\Qj \in \BL(\XX;\R^J)$ is orthonormal;
\State Set $\Ann = \cov \Qj (\Qj^\top \cov \Qj)^{-1/2}$;
\end{algorithmic}
\label{alg:nystroem}
\end{algorithm}

It has been shown that this method leads to a stochastic low-rank approximation of optimal order.

\begin{theorem}[Nyström low rank approximation]\label{nyström}
Let $\Ann$ be obtained from \autoref{alg:nystroem} and let $(\lambda_n)$ denote the decreasing eigenvalues of $\cov$. Then
\begin{equation}\label{eq:nystroem_precise}
\Exp{\xopnorm{\Ann \Ann^* - \cov}} \leq \left(1 + \sqrt{\frac{J}{J-N-1}}\right) \lambda_{N+1} + \frac{e \sqrt{2J - N}}{J-N} \sqrt{ \sum_{n > N} \lambda_n^2},
\end{equation}
for all $N \in \N$ with $N \leq J-2$, where $e$ denotes Euler's number. In particular, if \autoref{eigen_decay} is satisfied with $\eta>1/2$, we have
\begin{equation}\label{eq:nystroem_order}
\Exp {\xopnorm{\Ann \Ann^* - \cov}} = O(J^{-\eta}).
\end{equation}
\end{theorem}

\begin{proof}
It follows from lemma 4 in \cite{DriMah05} that
\begin{align*}
\xopnorm{\Ann \Ann^* - \cov} \leq \xopnorm{\Qj \Qj^* \cov - \cov},
\end{align*}
where $\Qj$ is as in algorithm 1. The right-hand side can be estimated using \cite[theorem 10.6]{HalMarTro11} (the adaptation to our infinite-dimensional setting is straightforward), yielding \autoref{eq:nystroem_precise}. If we then choose $N = J/2$ in \autoref{eq:nystroem_precise} (assuming without loss of generality that $J$ is even), the right-hand side becomes
\begin{align*}
\left(1 + \frac{J}{J/2-1}\right) \lambda_{J/2+1} + \frac{e \sqrt{3J/2}}{J/2} \sqrt{ \sum_{n > J/2} \lambda_n^2} \leq O(J^{-\eta}) + \frac{e \sqrt{3J/2}}{J/2} O(J^{-\eta + 1/2}) = O(J^{-\eta}).
\end{align*}
\end{proof}

\begin{remark}
By adapting the proof of \cite[theorem 10.6]{HalMarTro11}, one could also show that the Nyström-method is of $p$-order $\eta$, for all $p \in [1,\infty)$.
\end{remark}

We will see in \autoref{sec:numerics} that the accuracy of the Nyström method is very close to the theoretical optimum given by the truncated singular value decomposition.

\subsection{Convergence of direct EKI}\label{subsec:tikhonov}

The ensemble anomaly, truncated singular value decomposition, Nyström method, or in fact any other method for the low-rank approximation of positive operators can be used inside EKI. The corresponding error estimates with respect to Tikhonov regularization follow then directly from \autoref{large_ensemble_convergence}.

\begin{corollary}\label{eki_convergence}
Suppose that \autoref{assumption} is satisfied. Then:
\begin{enumerate}
\item Let $\Aj = \Anom(\ens)$. If $\cov$ is in the trace-class, then for all $p \in [1,\infty)$ there exists a constant $\kappa_p^\mathrm{en}$ such that
\begin{equation}\label{eq:standard_eki_convergence}
\Lpxnorm{\Xha - \xtika}{p} \leq \kappa_p^\mathrm{en} \phi(\alpha) J^{-1/2} \qquad \text{for all } \alpha > 0.
\end{equation}
\item Let $\Aj = \Ant$. Then $\Xha$ is deterministic, and if \autoref{eigen_decay} holds, then there exists a constant $\kappa^\mathrm{svd}$ such that
\begin{equation}\label{eq:svd_eki_convergence}
\xnorm{\Xha - \xtika} \leq \kappa^\mathrm{svd} \phi(\alpha) J^{-\eta} \qquad \text{for all } \alpha > 0.
\end{equation}
\item Let $\Aj = \Ann$. If \autoref{eigen_decay} holds with $\eta > 1/2$, then there exists a constant $\kappa^\mathrm{nys}$ such that
\begin{equation}\label{eq:nys_eki_convergence}
\Exp{ \xnorm{\Xha - \xtika} } \leq \kappa^\mathrm{nys} \phi(\alpha) J^{-\eta} \qquad \text{for all } \alpha > 0.
\end{equation}
\end{enumerate}
\end{corollary}

\begin{proof}
Let $p \in [1,\infty)$. By \autoref{sample_covariance}, $(\Anom(\ens))_{J=1}^\infty$ generates a stochastic low-rank approximation of $p$-order $1/2$. Thus, \autoref{eq:standard_eki_convergence} follows from \autoref{large_ensemble_convergence}. The estimates \autoref{eq:svd_eki_convergence} and \autoref{eq:nys_eki_convergence} then follow analogously through \autoref{schmidt} and \autoref{nyström}, respectively.
\end{proof}

\section{Adaptive ensemble Kalman inversion}
\label{sec:rates}

We have seen in \autoref{large_ensemble_convergence} that direct ensemble Kalman inversion can be understood as a low-rank approximation of Tikhonov regularization. It is well-known that, under a standard source-condition (see \autoref{source_condition} below), the Tikhonov-regularized solution of a linear equation converges to the infinite-dimensional minimum-norm solution (see \autoref{de:minimum_norm_solution}) in the zero-noise limit with a certain rate (see e. g. \cite{EngHanNeu96}). Thus, if we ensure that the error between direct EKI and Tikhonov regularization vanishes with the same rate as Tikhonov regularization converges, then direct EKI will also converge with that rate. However, since its iterates are restricted to the finite-dimensional range of $\Aj$, direct EKI can only lead to a convergent regularization method if the sample size $J$ is adapted to the noise level.

In this section, we describe how this can be achieved in conjunction with the discrepancy prinicple. The resulting method, which we call \emph{adaptive ensemble Kalman inversion}, is a convergent regularization method of optimal order in a sense that will be given below.
For this result, we require knowledge of a number $\delta > 0$ such that
\begin{align}
\noicovnorm{\obs - y} \leq \delta. \label{eq:noise_level}
\end{align}
This assumption is often referred to as a \emph{deterministic} noise model, and the number $\delta$ is called the \emph{deterministic noise level}. For some results on regularization with random noise, see for example \cite{BisHohMunRuy07}.

We start with a precise description of the adaptive EKI method in \autoref{sec:adaptive_description}, followed by a convergence analysis of the zero-noise limit in \autoref{sec:adaptive_analysis}. General remarks explaining the connection to other forms of EKI and multiscale methods are given in \autoref{sec:general_remarks}.

\subsection{Description of the method}\label{sec:adaptive_description}

We start by presenting a version of direct EKI with \emph{a-posteriori} parameter choice rule in the form of the discrepancy principle. We will refer to this method as \emph{adaptive EKI}.

In our definition, we distinguish between the cases where the underlying low-rank approximation is deterministic and stochastic. In the stochastic case, we will use a projection onto a suitably large ball around the initial guess $x_0$. This projection serves to guarantee stability of the resulting iteration even in the presence of non-deterministic sampling error. In \autoref{sec:adaptive_analysis}, we will see that if the radius of the ball is chosen sufficiently large, it does not negatively affect the convergence behavior.

\begin{definition}[Adaptive EKI]\label{de:adaptive_eki}
Let $\gamma > 0$, $b \in (0,1)$, $\alpha_0 > 0$ and $J_0 \in \N$, and define
\begin{align}
\alpha_k & = b^k \alpha_0, \label{eq:alpha_k}\\
\text{and} \quad J_k & = \lceil b^{-\frac{2k}{\gamma}} J_0 \rceil. \label{eq:j_k}
\end{align}
\begin{itemize}
\item If $(\Aj)_{J=1}^\infty$ generates a deterministic low-rank approximation of $\cov$, of order $\gamma$, we define the \emph{adaptive EKI} iteration associated to $(\Aj)_{J=1}^\infty$ as
\begin{align}
\Xaddet_k := \Xhvar{J_k}_{\alpha_k}, \qquad \text{for } k \in \N, \label{eq:adaptive_deterministic_EKI}
\end{align}
where $\Xhvar{J_k}_{\alpha_k}$ is defined in \autoref{de:direct}. If $\Aj = \Ant$ (see \autoref{schmidt}), we refer to the method as \emph{adaptive SVD-EKI} and denote its iterates with $\Xasvd_k$.
\item Let $r > 0$ and let $\xoball$ denote the closed ball around $x_0$ with radius $r$. Let $\projball$ denote the orthogonal projection on $\xoball$. If $(\Aj)_{J=1}^\infty$ generates a stochastic low-rank approximation of $\cov$, of $p$-order $\gamma$, we define the \emph{adaptive EKI} iteration associated to $(\Aj)_{J=1}^\infty$ as
\begin{align}
\Xad_k(\omega) := \projball \left(\Xhvar{J_k}_{\alpha_k}(\omega) \right), \qquad \text{for } k \in \N \text{ and } \omega \in \Omega, \label{eq:adaptive_EKI}
\end{align}
where $\Xhvar{J_k}_{\alpha_k}$ is defined in \autoref{de:direct}.
If $\Aj = \Anom(\ens)$ (see \autoref{de:anomaly}), we refer to this method as \emph{adaptive Standard-EKI} and denote its iterates with $\Xaeki_k$. Similarly, if $\Aj = \Ann$ (see \autoref{nyström}), we refer to the method as \emph{adaptive Nyström-EKI} and denote its iterates with $\Xanys_k$.
\end{itemize}
\end{definition}

The exponential reduction of the regularization parameter, given by \autoref{eq:alpha_k}, is a typical choice for regularization methods of similar form, and can already be found in \cite{Bak92}. The choice of $(J_k)_{k=1}^\infty$ is motivated by \autoref{large_ensemble_convergence}: By ensuring that $J_k^\gamma$ grows at least as fast as $\alpha_k^{-1}$, we make sure that the approximation error between adaptive EKI and Tikhonov regularization does not explode as $k$ increases.

In order to ensure convergence, we choose a stopping criterion for the adaptive EKI iteration. We consider the discrepancy principle, which has the advantage that it is easy to implement and it requires only little prior information on the forward operator $L$. In the case where the employed low-rank approximation is stochastic, the resulting stopping index is a random variable.

\begin{definition}[Discrepancy principle]\label{de:discrepancy}
Let $\delta$ be as in \autoref{eq:noise_level} and $\tau > 1$. Then, adaptive EKI (\autoref{de:adaptive_eki}) is terminated after $K_\delta$ iterations, where the integer random variable $K_\delta: \Omega \to \N \cup \{\infty \}$ satisfies
\begin{align}
\noicovnorm{\obs - L \Xad_{K_\delta(\omega)}(\omega) }
\leq \tau \delta < \noicovnorm{\obs - L \Xad_k(\omega) } \qquad \text{for all } k < K_\delta(\omega), \label{eq:ekiDiscrepancy}
\end{align}
where we set $K_\delta(\omega) = \infty$ if such a number does not exist.
\end{definition}

For the case of Tikhonov regularization, it is known that the discrepancy principle yields a converging regularization method under standard assumptions. The main difficulty of the analysis of adaptive EKI is to show that this result also holds for the random, approximate iteration given by \autoref{de:adaptive_eki}.

Pseudo-code for the adaptive EKI method in conjunction with the discrepancy principle is given in \autoref{alg:adaptive_eki}.

\begin{algorithm}[H]
\caption{\texttt{Adaptive EKI}}
Given $\obs$, $L$, $\delta$, $\alpha_0 > 0$, $J_0 \in \N$, $\tau > 1$, $b \in (0,1)$, $r > 0$, and deterministic or stochastic low-rank approximation $(\Aj)$ of $\cov$, of order $\gamma$.
\begin{algorithmic}[2]
\For{$k=1,\ldots$}
\State set $\alpha_k = b^k \alpha_0$;
\State set $J_k = \lceil b^{-\frac{k}{\gamma}} J_0 \rceil$;
\State set $\bm A_k = \bm A^{\scriptscriptstyle (J_k)}$;
\State compute $\bm B_k = \noicov^{-1/2} L \bm A_k$ by applying $\noicov^{-1/2} L$ to all columns of $\bm A_k$;
\State set $ \Xad_k = \xo + \bm A_k \left(\bm B_k^* \bm B_k  + \alpha_k \Ij \right)^{-1} \bm B_k^* R^{-1/2} (\obs - L \xo)$;
\If{$(\Aj)$ is stochastic and $\xnorm{\Xad_k - x_0} > r$}
\State set $\Xad_k = \projball(\Xad_k)$;
\EndIf
\If{$\norm{\obs - L \Xad_k}_\noicov \leq \tau \delta$}
\State break;
\EndIf
\EndFor
\State return $\Xad_k$;
\end{algorithmic}
\label{alg:adaptive_eki}
\end{algorithm}

\subsection{Convergence analysis}\label{sec:adaptive_analysis}

Next, we show that adaptive EKI as defined above is a convergent regularization method, where convergence is considered relative to the minimum-norm solution of \autoref{eq:inverseProblem}, defined as follows.

\begin{definition}\label{de:minimum_norm_solution}
We call $x^\dagger \in \XX$ an $(x_0, \cov)$-minimum-norm solution of $L x = y$ if
\begin{align*}
x^\dagger \in \argmin_{x \in \XX}\Set{\covnorm{x - x_0}}{Lx = y}.
\end{align*}
\end{definition}
The existence and uniqueness of $x^\dagger$ follow from \cite[theorem 2.5]{EngHanNeu96} taking into account \autoref{c_space}.

Before we continue, it is convenient to summarize the different inversion techniques and the according terminology.
\begin{center}
\begin{tabular}{||r||c|l||}
\hline
& Random variable & \\
\hline
$\Xk$ & $k$-th iterate of EKI with sample size $J$ & \autoref{eq:entik1}\\
$\Xha$ & Direct EKI with regularization parameter $\alpha$ & \autoref{eq:eki_alpha1}\\
$\Xaddet_k$ & \vtop{\hbox{\strut The $k$-th iterate of adaptive EKI with a deterministic}\hbox{\strut low-rank approximation}} & \autoref{eq:adaptive_deterministic_EKI}\\
$\Xad_k$ & \vtop{\hbox{\strut The $k$-th iterate of adaptive EKI with a stochastic}\hbox{\strut low-rank approximation}} & \autoref{eq:adaptive_EKI}\\
\hline
$\xtik_\alpha$ & Tikhonov-regularized solution according to the noisy data $\obs$& \autoref{eq:tikhonov_formula}\\
$x_\alpha$ & Tikhonov-regularized solution according to the exact data $y$& \autoref{eq:unperturbed_tikhonov} \\
\hline
\hline
\end{tabular}
\end{center}
 
Our convergence proof is based on the assumption that $x^\dagger$ satisfies a source condition, which is defined as follows.

\begin{assumption}[Source condition]\label{source_condition}
Let $\XX_{\cov}$ be defined as in \autoref{c_space}. There exists a $(\xo, \cov)$-minimum-norm solution $\xmns \in \XX_{\cov}$ of $Lx = y$, constants $\mu \in (0,1/2]$, $\rho > 0$ , and some $v \in \XX$ with
$\xnorm{v} \leq \rho$ such that
\begin{align} \label{eq:standardSourceCondition}
\xmns - \xo = C_0^{1/2} (B^* B)^\mu v,
\end{align}
where $B = R^{-1/2} L C_0^{1/2}$.
\end{assumption}
\begin{remark}
\autoref{eq:standardSourceCondition} can be interpreted as a smoothness assumption on the minimum-norm solution $\xmns$. Source conditions are ubiquitous in the mathematical literature on inverse problems. Typically, convergence rates for regularization methods cannot be proven without assuming some type of source condition. Beyond the condition \autoref{eq:standardSourceCondition}, also logarithmic, variational, and spectral tail conditions can be considered. See \cite{Gro83,Neu97,Hoh00,Sch01a} and some more recent references \cite{AndElbHooQiuSch15,AlbElbHooSch16}.
\end{remark}

For the subsequent convergence analysis, we focus first on the more challenging case where adaptive EKI is based on a \emph{stochastic} low-rank approximation. In that case, the following additional assumptions are sufficient to obtain convergence rates.

\begin{assumption}\label{parameter_assumptions}
Let $p, q \in [1, \infty)$, $\epsilon \in (0, \tau -1)$, and let $(\Aj)_{J=1}^\infty$ generate a stochastic low-rank approximation of $\cov$, of $p$-order $\gamma$.
\begin{enumerate}
\item The projection radius $r$ from \autoref{de:adaptive_eki} satisfies
\begin{align}
r \geq 2 \xopnorm{\cov^{1/2}} \covnorm{x_0 - x^\dagger}. \label{eq:r_large_enough}
\end{align}
\item There holds
\begin{align}
\phi(\alpha_0)^{-1} J_0^\gamma \geq \frac{c_{RL} \kappa_p}{\epsilon \delta^{1+\frac{q}{p}}}, \label{eq:initial_size_stochastic}
\end{align}
where $c_{RL}$ is as in \autoref{eq:rl_condition} and $\kappa_p$ is as in \autoref{large_ensemble_convergence}.
\end{enumerate}
\end{assumption}

\begin{remark}\label{re:adapted_size}
Note that \autoref{eq:initial_size_stochastic} together with \autoref{eq:alpha_k} and \autoref{eq:j_k} implies that a corresponding estimate holds for all subsequent iterates, i.e.
\begin{align}
\phi(\alpha_k)^{-1} J_k^\gamma \geq b^{2k} \phi(\alpha_0)^{-1} (b_k^{-\frac{2 k}{\gamma}})^\gamma J_0^\gamma \geq \frac{c_{RL} \kappa_p}{\epsilon \delta^{1+\frac{q}{p}}} \qquad \text{for all } k \in \N. \label{eq:adapted_size_stochastic}
\end{align}
Furthermore, the condition given by \autoref{eq:r_large_enough} simply means that the projection radius $r$ has to be chosen large enough in relation to the initial error $\covnorm{x_0 - x^\dagger}$. We show in \autoref{pr:adaptive_eki_error} that this condition ensures that the projection in \autoref{eq:adaptive_EKI} does not increase the approximation error between adaptive EKI and Tikhonov regularization.
\end{remark}

Our strategy to obtain convergence rates for adaptive EKI is to use the error estimate between direct EKI and Tikhonov regularization, provided by \autoref{large_ensemble_convergence}, to transfer the well-established convergence results on Tikhonov regularization to adaptive EKI. The main complication is that the discrepancy principle introduces a coupling between the regularization parameter and the sampling error, which makes it challenging to estimate $\xnorm{\Xad_{K_\delta} - \xtik_{\alpha_{K_\delta}}}$ directly. Instead, we employs a good-set strategy, similar to the one used in \cite{BauHohMun09} for the analysis of the iteratively regularized Gauss-Newton method for random noise. The idea behind the good-set-strategy is to define a suitable subset of $\Omega$ on which we can perform a deterministic analyis, and then to show that the probability of the complement vanishes sufficiently fast. For our purpose, we define the good set $\goodset \subset \Omega$ by
\begin{align}
& \goodset := \bigcap_{k=1}^{k_\delta} \goodsetk, \label{eq:goodset}\\
\text{where} \quad & \goodsetk := \Set{\omega \in \Omega}{\xnorm{\Xad_k(\omega) - \xtik_{\alpha_k}} \leq c_{RL}^{-1} \epsilon \delta} \label{eq:goodsetk} \\
\text{and} \quad & k_\delta := \min \Set{k \in \N}{ \noicovnorm{\obs - L \xtik_{\alpha_k}} \leq (\tau - \epsilon) \delta}.
\end{align}
Then, the law of total expectation yields, for $q \in [1, \infty)$,
\begin{align*}
\Exp{\xnorm{\Xad_{K_\delta} - x^\dagger}^q} &= \Exp{\xnorm{\Xad_{K_\delta} - x^\dagger}^q | \goodset} \PP(\goodset) + \Exp{\xnorm{\Xad_{K_\delta} - x^\dagger}^q | \goodset^\complement} \PP(\goodset^\complement) \\
& \leq \Exp{\xnorm{\Xad_{K_\delta} - x^\dagger}^q | \goodset} + \Exp{\xnorm{\Xad_{K_\delta} - x^\dagger}^q | \goodset^\complement} \PP(\goodset^\complement),
\end{align*}
where $\goodset^\complement$ denotes the complement of $\goodset$. Since $\Xad_{K_\delta} \in \xoball$ holds by \autoref{eq:adaptive_EKI}, we have
\begin{align}
\Exp{\xnorm{\Xad_{K_\delta} - x^\dagger}^q} \leq \Exp{\xnorm{\Xad_{K_\delta} - x^\dagger}^q | \goodset} + (r + \xnorm{x^\dagger})^q \PP(\goodset^\complement). \label{eq:splitting_strategy}
\end{align}
Hence, it suffices to estimate $\Exp{\xnorm{\Xad_{K_\delta} - x^\dagger}^q | \goodset}$ and $\PP(\goodset^\complement)$ separately.

Estimates for the first term hinge on understanding the behavior of the Tikhonov-regularized solution $\xtik_{\alpha_{k_\delta}}$. The following lemma summarizes existing results on Tikhonov regularization that we will make use of in our theoretical analysis of adaptive EKI. To this end, we consider as an auxiliary variable the Tikhonov-regularized solution of \autoref{eq:inverseProblem} according to the \emph{exact} data $y$ and regularization parameter $\alpha$, defined as
\begin{align}
x_\alpha := \xo + \cov^{1/2} \left(\cov^{1/2} L^* R^{-1} L \cov^{1/2} + \alpha \Idx \right)^{-1} \cov^{1/2} L^* R^{-1} (y - L \xo). \label{eq:unperturbed_tikhonov}
\end{align}
(Compare \autoref{de:tikhonov}.)

\begin{lemma}[Convergence and stability of Tikhonov regularization]\label{det_alpha}
Suppose that \\
\autoref{assumption} and \autoref{source_condition} hold, and let $x_0 \in \XX$, $\alpha > 0$ and $\delta > 0$. Moreover, assume that
\begin{align}
\noicovnorm{\obs - y} \leq \delta. \label{eq:delta}
\end{align}
Then there holds
\begin{align}
\covnorm{\xa - \xmns} & \leq \rho^\frac{1}{2\mu + 1} \noicovnorm{L x_\alpha - y}^\frac{2\mu}{2 \mu + 1}, \label{eq:det_alpha1}\\
\text{and} \quad \noicovnorm{L(\xa - \xtika) - (y - \obs)} & \leq \delta. \label{eq:det_alpha2}
\end{align}
Furthermore, there exist constants $c_1$ and $c_2$, independent of $\alpha$, $\delta$ and $\rho$, such that
\begin{align}
\covnorm{\xa - \xtika} & \leq c_1 \delta \alpha^{-1/2}, \label{eq:det_alpha3}\\
\text{and} \quad \noicovnorm{y - L \xa} & \leq c_2 \rho \alpha^{\mu + 1/2} \label{eq:det_alpha4}.
\end{align}
\end{lemma}

\begin{proof}
Note that $\xmns$ is a $(\xo, \cov)$-minimum-norm solution of \autoref{eq:inverseProblem} if and only if $\xmns = \xo + \cov^{1/2} w^\dagger$, where $w^\dagger$ is a $(0,\Idx)$-minimum-norm solution of
\begin{align}
R^{-1/2} y = B w, \label{eq:transformed_problem}
\end{align}
where $B := R^{-1/2} L \cov^{1/2}$. Similary, if $w_\alpha$ is the corresponding Tikhonov-regularized solution of \autoref{eq:transformed_problem}, i.e.
\begin{align*}
w_\alpha = \left(B^*B + \alpha \Idx \right)^{-1} B^*R^{-1/2}y,
\end{align*}
then $x_\alpha = \xo + \cov^{1/2} w_\alpha$. Thus, the results follow from the classical case where $\cov = \Idx$ and $\noicov = \Idy$: The inequalities \autoref{eq:det_alpha1}, \autoref{eq:det_alpha2} and \autoref{eq:det_alpha3} can be found in \cite[(4.66)]{EngHanNeu96}, \cite[(4.68)]{EngHanNeu96} and \cite[(4.70)]{EngHanNeu96},
respectively. \autoref{eq:det_alpha4} can be obtained from the source condition \autoref{eq:standardSourceCondition} and the interpolation inequality \cite[(4.64)]{EngHanNeu96}, as in the proof of \cite[theorem 4.17]{EngHanNeu96}.
\end{proof}

Moreover, for the deterministic stopping time $k_\delta$ the following auxiliary result holds.

\begin{lemma}\label{lem:deterministic_stopping}
Given \autoref{assumption}, \autoref{source_condition}.
\begin{enumerate}
\item There exists a constant $c_3$, independent of $\alpha$, $\delta$, and $\rho$, such that
\begin{align}
\alpha_{k_\delta} \geq c_3 \left(\frac{\delta}{\rho} \right)^\frac{2}{2 \mu + 1} \label{eq:alpha_estimate}
\end{align}
for all sufficiently small $\delta$.
\item There holds
\begin{align}
k_\delta = O \left( \log(\delta^{-1}) \right). \label{eq:k_delta_estimate}
\end{align}
\item If \autoref{eq:r_large_enough} holds, then there exists a sufficiently small $\bar{\delta} > 0$ such that
\begin{align}
\xtik_{\alpha_k} \in \xoball \qquad \text{for all } k \leq k_\delta \text{ and all } \delta \leq \bar{\delta}. \label{eq:tikhonov_boundedness}
\end{align}
\end{enumerate}
\end{lemma}

\begin{proof}
\begin{enumerate}
\item Using the same transformations as in the proof of \autoref{det_alpha}, the statement follows from the proof of theorem 4.17 in \cite{EngHanNeu96}. Note that this proof uses a discrepancy principle where $\alpha$ can vary continuously. However, the same argument applies also to the discretized sequence satisfying \autoref{eq:alpha_k}, see \cite[remark 4.18]{EngHanNeu96}.
\item Inserting \autoref{eq:alpha_k} in \autoref{eq:alpha_estimate} yields
\begin{align*}
b^{k_\delta} \alpha_0 \geq c_3 \left(\frac{\delta}{\rho} \right)^\frac{2}{2 \mu + 1},
\end{align*}
or equivalently
\begin{align*}
\left(\frac{1}{b}\right)^{k_\delta} \leq \frac{\alpha_0}{c_3} \left( \frac{\rho}{\delta} \right)^\frac{2}{2 \mu + 1}.
\end{align*}
Taking the logarithm and using the fact that $b \in (0,1)$, we arrive at
\begin{align*}
k_\delta \leq \log(b^{-1})^{-1} \cdot \left[ \log \left(\frac{\alpha_0}{c_3} \right) + \frac{2}{2 \mu + 1} \log \left( \frac{\rho}{\delta} \right) \right].
\end{align*}
This proves \autoref{eq:k_delta_estimate}.
\item As in the proof of \autoref{det_alpha}, let $B = R^{-1/2} L \cov^{1/2}$, $x^\dagger = x_0 + \cov w^\dagger$ and $\xtik_\alpha = \xo + \cov^{1/2} \wtik_\alpha$, such that
\begin{align*}
\wtik_\alpha = \left( B^* B + \alpha \Id_\XX \right)^{-1} B^* R^{-1/2}(\obs - L \xo).
\end{align*}
Then
\begin{align}
\wtik_\alpha &= \left( B^* B + \alpha \Id_\XX \right)^{-1} B^* R^{-1/2}(\obs - L \xo) \notag \\
& = \left( B^* B + \alpha \Id_\XX \right)^{-1} B^* R^{-1/2}(\obs - y) + \left( B^* B + \alpha \Id_\XX \right)^{-1} B^* R^{-1/2}(y - L x_0). \label{eq:wtik_wdagger_equation}
\end{align}
If we insert $y = L x^\dagger = L (x_0 + \cov^{1/2} w^\dagger)$ in the second term on the right-hand side of \autoref{eq:wtik_wdagger_equation}, we obtain after cancellation and using the definition of $B$,
\begin{align*}
\wtik_\alpha - w^\dagger &= \left(B^* B + \alpha \Id_\XX \right)^{-1} B^* R^{-1/2}(\obs - y) + \left(B^* B + \alpha \Id_\XX \right)^{-1} B^* B w^\dagger. \label{eq:wtik_equation2}
\end{align*}
Using \autoref{eq:noise_level} and the spectral estimates (see e.g. \cite[lemma 4.5]{KalNeuSch08})
\begin{align*}
\xopnorm{\left( B^* B + \alpha \Id_\XX \right)^{-1} B^*} & \leq \frac{1}{2}\alpha^{-1/2}, \\
\text{and} \quad \xopnorm{\left( B^* B + \alpha \Id_\XX \right)^{-1} B^* B} & \leq 1,
\end{align*}
in \autoref{eq:wtik_wdagger_equation}, we obtain
\begin{align}
\xnorm{\wtik_\alpha - w^\dagger} \leq \frac{1}{2} \alpha^{-1/2} \delta + \xnorm{w^\dagger}. \label{eq:wtik_estimate}
\end{align}
Finally, it follows from \autoref{eq:alpha_estimate} that
\begin{align}
\alpha^{-1/2}_k \delta \leq \alpha^{-1/2}_{k_\delta} \delta \leq c_3 \left(\frac{\rho}{\delta} \right)^\frac{1}{2 \mu + 1} \delta = c_3 \rho^\frac{1}{2 \mu + 1} \delta^\frac{2}{2 \mu + 1} \qquad \text{for all } k \leq k_\delta, \label{eq:alpha_delta_estimate}
\end{align}
which vanishes as $\delta \to 0$. Therefore, if we set 
\begin{align*}
\bar \delta = 2^\frac{2 \mu + 1}{2} c_3^{-\frac{2 \mu + 1}{2}} \rho^{-\frac{1}{2}} \xnorm{w^\dagger}^\frac{2 \mu + 1}{2},
\end{align*}
then it follows from \autoref{eq:wtik_estimate} and \autoref{eq:alpha_delta_estimate} that
\begin{align}
\xnorm{\wtik_{\alpha_k}} \leq 2 \xnorm{w^\dagger} \qquad \text{for all } k \leq k_\delta \label{eq:wtik_less_2wdagger}
\end{align}
holds for all $\delta \leq \bar \delta$. By definition of $\wtik_{\alpha_k}$, \autoref{eq:wtik_less_2wdagger} implies
\begin{align*}
\covnorm{\xtik_{\alpha_k} - x_0} = \xnorm{\wtik_{\alpha_k}} \leq 2 \xnorm{w^\dagger} = 2 \covnorm{x^\dagger - x_0},
\end{align*}
and hence, by \autoref{eq:c_norm_stronger} and \autoref{eq:r_large_enough},
\begin{align*}
\xnorm{\xtik_{\alpha_k} - x_0} \leq \xopnorm{\cov^{1/2}} \covnorm{\xtik_{\alpha_k} - x_0} \leq 2 \xopnorm{\cov^{1/2}} \covnorm{x^\dagger - x_0} \leq r,
\end{align*}
for all $k \leq k_\delta$ and $\delta \leq \bar{\delta}$.
\end{enumerate}
\end{proof}

With this lemma, we are able to show that the projection in \autoref{eq:adaptive_EKI} cannot increase the approximation error between adaptive EKI and the corresponding Tikhonov iteration, at least for $k \leq k_\delta$. More precisely, we have the following proposition.

\begin{proposition}\label{pr:adaptive_eki_error}
Let \autoref{assumption}, \autoref{source_condition} and \autoref{eq:r_large_enough} hold. Let $\delta \leq \bar \delta$, where $\bar \delta$ is as in \autoref{lem:deterministic_stopping}. Then
\begin{align}
\xnorm{\Xad_k(\omega) - \xtik_{\alpha_k}} \leq \xnorm{\Xhvar{J_k}_{\alpha_k}(\omega) - \xtik_{\alpha_k}} \qquad \text{for all } \omega \in \Omega \text{ and all } k \leq k_\delta. \label{eq:projection_not_worse}
\end{align}
In particular,
\begin{align}
\Lpxnorm{\Xad_k - \xtik_{\alpha_k}}{p} \leq \kappa_p \phi(\alpha_k) J_k^{- \gamma} \qquad \text{for all } k \leq k_\delta, \label{eq:adaptive_eki_converges}
\end{align}
where $\kappa_p$ is as in \autoref{large_ensemble_convergence}.
\end{proposition}

\begin{proof}
Let $k \leq k_\delta$ and $\omega \in \Omega$. By \autoref{eq:adaptive_EKI}, we have
\begin{align}
\xnorm{\Xad_k(\omega) - \xtik_{\alpha_k}} = \xnorm{\projectedeki - \xtik_{\alpha_k}}. \label{eq:consider_the_projection}
\end{align}
By \autoref{lem:deterministic_stopping}, we have $\xtik_{\alpha_k} \in \xoball$. Consequently, by the property of the orthogonal projection, we have
\begin{align}
\inner{\Xhvar{J_k}_{\alpha_k}(\omega) - \projectedeki}{ \xtik_{\alpha_k} - \projectedeki} \leq 0. \label{eq:projection_property}
\end{align}
This yields
\begin{align*}
\xnorm{\Xhvar{J_k}_{\alpha_k}(\omega) - \xtik_{\alpha_k}}^2 & = \xnorm{\Xhvar{J_k}_{\alpha_k}(\omega) - \projectedeki + \projectedeki - \xtik_{\alpha_k}}^2 \\
& = \xnorm{\Xhvar{J_k}_{\alpha_k}(\omega) - \projectedeki}^2 + \xnorm{\projectedeki - \xtik_{\alpha_k}}^2 \\
& \qquad + 2 \inner{\Xhvar{J_k}_{\alpha_k}(\omega) - \projectedeki}{\projectedeki - \xtik_{\alpha_k}} \\
& \geq \xnorm{\projectedeki - \xtik_{\alpha_k}}^2.
\end{align*}
Together with \autoref{eq:consider_the_projection}, this yields \autoref{eq:projection_not_worse}. \autoref{eq:adaptive_eki_converges} then follows from \autoref{eq:projection_not_worse} and \autoref{large_ensemble_convergence}.
\end{proof}

The next proposition provides the desires asymptotic convergence rates of the probability $\PP(\goodset^\complement)$.

\begin{proposition}\label{complement_estimate}
Given \autoref{assumption}, \autoref{source_condition} and \autoref{parameter_assumptions}, there holds
\begin{equation}
\PP(\goodset^\complement) = O(\delta^\frac{2 \mu q}{2 \mu + 1}). \label{eq:complement_estimate}
\end{equation}
\end{proposition}

\begin{proof}
By \autoref{eq:goodset} and the subadditivity of $\PP$, we have
\begin{align}
\PP(\goodset^\complement) = \PP(\bigcup_{k=1}^{k_\delta}(\goodsetk)^\complement) \leq \sum_{k=1}^{k_\delta} \PP((\goodsetk)^\complement). \label{eq:goodset_estimate}
\end{align}
By \autoref{eq:goodsetk} and Markov's inequality (see \autoref{markov}), we have
\begin{align}
\PP((\goodsetk)^\complement) & = \PP\left(\Set{\omega \in \Omega}{\xnorm{\Xad_k - \xtik_{\alpha_k}} > c_{RL}^{-1} \epsilon \delta} \right) \notag\\
& \leq \frac{c_{RL}^p \Exp{\xnorm{\Xad_k - \xtik_{\alpha_k}}^p}}{\epsilon^p \delta^p}. \label{eq:goodsetk_estimate1}
\end{align}
Without loss of generality, let $\delta \leq \bar \delta$, where $\bar \delta$ is as in \autoref{lem:deterministic_stopping}. Using \autoref{pr:adaptive_eki_error} and then \autoref{eq:adapted_size_stochastic} in \autoref{eq:goodsetk_estimate1} yields
\begin{align*}
\PP((\goodsetk)^\complement) \leq \frac{c_{RL}^p \kappa_p^p \phi(\alpha_k)^p J_k^{-p \gamma}}{\epsilon^p \delta^p} \leq \delta^q.
\end{align*}
Inserting this inequality in \autoref{eq:goodset_estimate}, we arrive at
\begin{align}
\PP(\goodset^\complement) \leq \sum_{k=1}^{k_\delta} \delta^q = k_\delta \delta^q. \label{eq:goodset_estimate_k_delta}
\end{align}
From \autoref{lem:deterministic_stopping} we know that $k_\delta = O(\log(\delta^{-1}))$. Since we have $1 - \frac{2\mu}{2 \mu + 1} > 0$, we obtain
\begin{align}
k_\delta \delta^{(1 - \frac{2\mu}{2 \mu + 1})q} \to 0 \qquad \text{as } \delta \to 0. \label{eq:k_delta_majorized}
\end{align}
Hence, we have from \autoref{eq:goodset_estimate_k_delta} that
\begin{align*}
\PP(\goodset^\complement) \leq k_\delta \delta^{(1 - \frac{2\mu}{2 \mu + 1})q} \delta^\frac{2 \mu q}{2 \mu + 1} = O(\delta^\frac{2 \mu q}{2 \mu + 1}).
\end{align*}
\end{proof}

Finally, we show convergence of the random element $\Xad_{K_\delta}$ on the "good set" $\goodset$. The construction of $\goodset$ allows to apply the proof of \cite[theorem 4.17]{EngHanNeu96} with straightforward modifications to each individual realization $\Xad_{K_\delta(\omega)}(\omega)$, given $\omega \in \goodset$.

\begin{proposition}\label{convergence_on_good_set}
Given \autoref{assumption}, \autoref{source_condition} and \autoref{parameter_assumptions}, there exists $C > 0$, independent of $\omega$ and $\delta$, such that
\begin{align}
\xnorm{\Xad_{K_\delta(\omega)}(\omega) - x^\dagger} \leq C \delta^\frac{2 \mu}{2 \mu + 1} \qquad \text{for all } \omega \in \goodset. \label{eq:rate_on_good_set}
\end{align}
\end{proposition}

\begin{proof}
Let $\omega \in \goodset$.
\begin{itemize}
\item First, we show that $K_\delta(\omega) \leq k_\delta$: To see this, note that
\begin{align*}
\noicovnorm{\obs - L \Xad_{k_\delta}(\omega)} & \leq \noicovnorm{\obs - L \xtik_{\alpha_{k_\delta}}} + \noicovnorm{L \left( \xtik_{\alpha_{k_\delta}} - \Xad_{k_\delta}(\omega) \right)} \\
& \leq \noicovnorm{\obs - L \xtik_{\alpha_{k_\delta}}} + c_{RL} \covnorm{\xtik_{\alpha_{k_\delta}} - \Xad_{k_\delta}(\omega)}.
\end{align*}
By definition of $k_\delta$ and $\goodset$, this implies
\begin{align*}
\noicovnorm{\obs - L \Xad_{k_\delta}(\omega)} \leq (\tau - \epsilon) \delta + \epsilon \delta = \tau \delta.
\end{align*}
Hence, by definition of $K_\delta$, there must hold $K_\delta(\omega) \leq k_\delta$.
\item Since $K_\delta(\omega) \leq k_\delta$, we have by definition of $\goodset$,
\begin{align}
\xnorm{\Xad_k(\omega) - \xtik_k} \leq c_{RL}^{-1} \epsilon \delta \qquad \text{for all } k \leq \kdo, \label{eq:xnorm_estimate}
\end{align}
and consequently also
\begin{align}
\noicovnorm{L(\Xad_k(\omega) - \xtik_{\alpha_k})} \leq \epsilon \delta \qquad \text{for all } k \leq \kdo. \label{eq:noicovnorm_estimate}
\end{align}
\item Next, we show that there exists a constant $c_4$, independent of $\rho$, $\delta$ and $\omega$, such that
\begin{align}
\alpha_{\kdo}^{-1/2} \leq c_4 \left(\frac{\rho}{\delta}\right)^\frac{1}{2 \mu + 1}. \label{eq:akdo_estimate}
\end{align}
From \autoref{eq:det_alpha4}, we obtain
\begin{align}
\noicovnorm{y - L x_{\alpha_{\kdo - 1}}} &\leq c_2 \rho \alpha_{\kdo - 1}^{\mu + 1/2} \notag \\
& = c_2 \rho (b^{-1}\akdo)^{\mu + 1/2}. \label{eq:less_than_alpha}
\end{align}
On the other hand,
\begin{align*}
\noicovnorm{y - L x_{\alpha_{\kdo - 1}}} \geq \noicovnorm{\obs - L \xtik_{\alpha_{\kdo - 1}}} - \noicovnorm{(\obs - y) - L(\xtik_{\alpha_{\kdo - 1}} - x_{\alpha_{\kdo - 1}})}.
\end{align*}
Inserting \autoref{eq:det_alpha2} yields
\begin{align*}
\noicovnorm{y - L x_{\alpha_{\kdo - 1}}} & \geq \noicovnorm{\obs - L \xtik_{\alpha_{\kdo - 1}}} - \delta \\
& \geq \noicovnorm{\obs - L \Xad_{\kdo - 1}(\omega)} - \noicovnorm{L (\Xad_{\kdo - 1}(\omega) - \xtik_{\alpha_{\kdo - 1}})} - \delta.
\end{align*}
By the definition of $K_\delta$ and \autoref{eq:noicovnorm_estimate}, this reduces to
\begin{align}
\noicovnorm{y - L x_{\alpha_{\kdo - 1}}} \geq \tau \delta - \epsilon \delta - \delta = (\tau - \epsilon - 1) \delta. \label{eq:greater_than_delta}
\end{align}
Combining \autoref{eq:less_than_alpha} and \autoref{eq:greater_than_delta} yields
\begin{align*}
(\tau - \epsilon - 1) \delta \leq c_2 \rho (b^{-1}\akdo)^{\mu + 1/2}.
\end{align*}
Since $\tau - \epsilon - 1 > 0$, we can rearrange this inequality to
\begin{align*}
\akdo^{-1/2} \leq  b^{-1/2} \left(\frac{c_2}{\tau - \epsilon - 1} \right)^\frac{1}{2 \mu + 1}\left(\frac{\rho}{\delta}\right)^\frac{1}{2 \mu + 1},
\end{align*}
which shows \autoref{eq:akdo_estimate} for suitable choice of $c_4$.
\item Next, we show that there exists a constant $c_5$, independent of $\omega$ and $\delta$, such that
\begin{align}
\covnorm{\xtik_{\akdo} - x^\dagger} \leq c_5 \delta^\frac{2 \mu}{2 \mu + 1}. \label{eq:random_discrepancy_optimal}
\end{align}
We start with the triangle inequality
\begin{align}
\covnorm{\xtik_{\akdo} - x^\dagger} \leq \covnorm{\xtik_{\akdo} - x_{\akdo}} + \covnorm{x_{\akdo} - x^\dagger}. \label{eq:tikhonov_triangle}
\end{align}
By \autoref{eq:det_alpha3}, the first term on the right-hand side satisfies
\begin{align*}
\covnorm{\xtik_{\akdo} - x_{\akdo}} \leq c_1 \delta \akdo^{-1/2}.
\end{align*}
Inserting \autoref{eq:akdo_estimate} yields
\begin{align}
\covnorm{\xtik_{\akdo} - x_{\akdo}} \leq c_1 c_4 \rho^\frac{1}{2 \mu + 1} \delta^\frac{2 \mu}{2 \mu + 1}. \label{eq:noisy_vs_exact_tikhonov}
\end{align}
For the second term on the right-hand side of \autoref{eq:tikhonov_triangle}, we have by \autoref{eq:det_alpha1}:
\begin{align}
\covnorm{x_{\akdo} - x^\dagger} \leq \rho^\frac{1}{2 \mu + 1} \noicovnorm{L x_{\akdo} - y}^\frac{2 \mu}{2 \mu + 1}. \label{eq:exact_tikhonov_error1}
\end{align}
We then estimate, using \autoref{eq:det_alpha2},
\begin{align*}
\noicovnorm{L x_{\akdo} - y} & \leq \noicovnorm{\obs - L \xtik_{\akdo}} + \noicovnorm{ (y - \obs) - L \left( x_{\akdo} - \xtik_{\akdo} \right)} \\
& \leq \noicovnorm{\obs - L \xtik_{\akdo}} + \delta.
\end{align*}
From this, another use of the triangle inequality yields
\begin{align*}
\noicovnorm{L x_{\akdo} - y} \leq \noicovnorm{\obs - L \Xad_{\kdo}(\omega)} + \noicovnorm{L \left(\Xad_{\kdo} - \xtik_{\akdo} \right)} + \delta.
\end{align*}
Finally, using the definition of $K_\delta$ and \autoref{eq:noicovnorm_estimate} yields
\begin{align}
\noicovnorm{L x_{\akdo} - y} \leq (\tau + \epsilon + 1) \delta. \label{eq:exact_noicov_estimate}
\end{align}
Inserting \autoref{eq:exact_noicov_estimate} in \autoref{eq:exact_tikhonov_error1} yields
\begin{align}
\covnorm{x_{\akdo} - x^\dagger} \leq (\tau + \epsilon + 1) \rho^\frac{1}{2 \mu + 1} \delta^\frac{2 \mu}{2 \mu + 1}. \label{eq:exact_tikhonov_error2}
\end{align}
Finally, inserting both \autoref{eq:noisy_vs_exact_tikhonov} and \autoref{eq:exact_tikhonov_error2} in \autoref{eq:tikhonov_triangle} yields \autoref{eq:random_discrepancy_optimal} for sutable choice of $c_5$.
\item From the triangle inequality and \autoref{eq:c_norm_stronger}, we have
\begin{align*}
\xnorm{\Xad_{\kdo}(\omega) - x^\dagger} & \leq \xnorm{\Xad_{\kdo} - \xtik_{\akdo}} + \xnorm{\xtik_{\akdo} - x^\dagger} \\
& \leq \xnorm{\Xad_{\kdo} - \xtik_{\akdo}} + \xopnorm{\cov^{1/2}}\covnorm{\xtik_{\akdo} - x^\dagger}.
\end{align*}
We can use \autoref{eq:xnorm_estimate} to estimate the first and \autoref{eq:random_discrepancy_optimal} to estimate the second term of the right-hand side, which yields
\begin{align*}
\xnorm{\Xad_{\kdo}(\omega) - x^\dagger} = c_{RL}^{-1} \epsilon \delta + \xopnorm{\cov^{1/2}} \cdot c_5 \delta^\frac{2 \mu}{2 \mu + 1}.
\end{align*}
Hence, we can choose $C > 0$, independently of $\delta$ and $\omega$, such that \autoref{eq:rate_on_good_set} holds.
\end{itemize}
\end{proof}

With this, we arrive at convergence rates for adaptive EKI under a stochastic low-rank approximation.

\begin{theorem}\label{adaptive_eki_convergence_rate}
Given \autoref{assumption}, \autoref{source_condition} and \autoref{parameter_assumptions}, there holds
\begin{align}
\Lpxnorm{\Xad_{K_\delta} - x^\dagger}{q} = O( \delta^\frac{2 \mu}{2 \mu + 1}). \label{eq:adaptive_eki_convergence_rate}
\end{align}
\end{theorem}

\begin{proof}
By \autoref{convergence_on_good_set}, there holds
\begin{align*}
\xnorm{\Xad_{\kdo}(\omega) - x^\dagger}^q \leq C^q \delta^\frac{2 \mu q}{2 \mu + 1} \qquad \text{for all } \omega \in \goodset.
\end{align*}
This implies in particular
\begin{align*}
\Exp{\xnorm{\Xad_{K_\delta} - x^\dagger}^q | \goodset} \leq  C^q \delta^\frac{2 \mu q}{2 \mu + 1}.
\end{align*}
Using this inequality and \autoref{complement_estimate} in \autoref{eq:splitting_strategy} yields
\begin{align*}
\Exp{\xnorm{\Xad_{K_\delta} - x^\dagger}^q} = O(\delta^\frac{2 \mu q}{2 \mu + 1}),
\end{align*}
from which \autoref{eq:adaptive_eki_convergence_rate} follows.
\end{proof}

For completeness, we also formulate the convergence rate results under a deterministic low-rank approximation. In this case, the proof of \autoref{convergence_on_good_set} applies without change, and we obtain the following result.

\begin{theorem}\label{adaptive_deterministic_eki_convergence_rate}
Let \autoref{assumption} and \autoref{source_condition} hold, and let $(\Aj)_{J=1}^\infty$ generates a deterministic low-rank approximation of $\cov$, of order $\gamma$. Assume there is $\epsilon \in (0, \tau - 1)$ such that
\begin{align}
\phi(\alpha_0)^{-1} J_0^\gamma \geq \frac{c_{RL} \kappa}{\epsilon \delta}, \label{eq:adapted_size_deterministic}
\end{align}
where $\kappa$ is as in \autoref{large_ensemble_convergence}. Then
\begin{align*}
\xnorm{\Xaddet_{K_\delta} - x^\dagger} = O(\delta^\frac{2 \mu}{2 \mu + 1}).
\end{align*}
\end{theorem}

\begin{remark}\label{re:deterministic}
Comparing the condition \autoref{eq:adapted_size_deterministic} for the deterministic case to the condition \autoref{eq:adapted_size_stochastic} for the stochastic case, we see that the major difference is that the stochastic case requires an additional multiplicative factor $\delta^{-\frac{q}{p}}$. This additional factor is used in the proof of \autoref{complement_estimate} to ensure that $\PP((\goodset^\complement)) = O(\delta^\frac{2 \mu q}{2 \mu + 1})$. Formally, we recover the deterministic case from the stochastic case in the limit $p \to \infty$ (where $p=\infty$ corresponds to almost sure convergence).
\end{remark}

\begin{remark}\label{re:optimal}
The proven convergence rate is optimal for $\mu \in (0,\frac{1}{2})$, in the sense that if only \autoref{source_condition} is known, there exists no regularization method that satisfies a better general bound with respect to $\delta$ and $\mu$ \cite[proposition 3.15]{EngHanNeu96}.
\end{remark}

Continuing our discussion from \autoref{subsec:sampling}, we see from \autoref{adaptive_eki_convergence_rate} and \autoref{adaptive_deterministic_eki_convergence_rate} that the three special cases of adaptive EKI defined in \autoref{de:adaptive_eki}, namely adaptive Standard-, SVD- and Nyström-EKI, are all of (stochastic) optimal order. However, the faster convergence of the SVD- and Nyström-based low-rank approximation means that the sample size $J_k$ does not have to grow as fast as for Standard-EKI, which makes those two methods computationally cheaper.

\begin{corollary}\label{rates_special_cases}
Let \autoref{assumption} and \autoref{source_condition} hold.
\begin{enumerate}
\item Let $p \in [1,\infty)$, $\cov$ be in the trace class, and suppose that \autoref{parameter_assumptions} is satisfied for $\gamma=1/2$. Then there holds
\begin{align*}
\Lpxnorm{\Xaeki_{K_\delta} - \xmns}{p} = O( \delta^\frac{2 \mu}{2 \mu + 1}).
\end{align*}
\item Assume that $\cov$ satisfies \autoref{eigen_decay} with constant $\eta>0$, and suppose that \autoref{eq:adapted_size_deterministic} is satisfied for $\gamma = \eta$. Then there holds
\begin{align*}
\xnorm{\Xasvd_{K_\delta} - \xmns} = O( \delta^\frac{2 \mu}{2 \mu + 1}).
\end{align*}
\item Let $p \in [1,\infty)$, assume that $\cov$ satisfies \autoref{eigen_decay} with constant $\eta>1/2$, and suppose tat \autoref{parameter_assumptions} is satisfied for $\gamma=\eta$. Then there holds
\begin{align*}
\Exp {\xnorm{\Xanys_{K_\delta} - \xmns} } = O(\delta^\frac{2 \mu}{2 \mu + 1}).
\end{align*}
\end{enumerate}
\end{corollary}

\begin{proof}
Recall that $(\Xaeki_k)_{k=1}^\infty$ is a special case of adaptive EKI where the low-rank approximation is generated by $(\Anom(\ens))_{J=1}^\infty$ (see \autoref{sample_covariance}). Thus, if $\cov$ is in the trace-class, \autoref{adaptive_eki_convergence_rate} applies with $\gamma=1/2$ and yields the desired convergence rate. The corresponding result for $(\Xasvd_k)_{k=1}^\infty$ follows analogously from \autoref{adaptive_deterministic_eki_convergence_rate} and \autoref{schmidt}, while the result for $(\Xanys_k)_{k=1}^\infty$ follows from \autoref{adaptive_eki_convergence_rate} and \autoref{nyström}.
\end{proof}

As an example, suppose we know that $\cov$ is in the trace class, i.e. $\eta \geq 1$. Then \autoref{rates_special_cases} implies that Standard-EKI is of optimal order if $J_k \geq b^{-2 k} J_0$, whereas Nyström-EKI is of optimal order if $J_k \geq b^{- k} J_0$ (see \autoref{eq:j_k}). This means that Nyström-EKI performs comparably with only a square-root of the sample size. Furthermore, if the eigenvalues of $\cov$ decay faster than $O(n^{-1})$, Nyström-EKI can take advantage of this, whereas Standard-EKI is limited by the lower bound \autoref{eq:lower_bound}.

\subsection{General remarks}\label{sec:general_remarks}

\subsubsection*{Relation to other versions of EKI}

Note that our focus differs from the strictly Bayesian setting in which ensemble Kalman inversion is often introduced. In the Bayesian setting, it is assumed that the regularization parameter represents the available prior information, and the regularized solution is identified with the MAP estimate. In regularization theory, we are interested in showing convergence rates in the zero-noise limit, which requires the use of parameter choice rules that select the regularization parameter $\alpha$ in terms of the noise level and properties of the forward operator $L$. Above, we have focused on the discrepancy principle. In contrast to a-priori choice rules, the use of the discrepancy principle has the advantage that it requires only little prior information on the operator $L$. However, its use is contingent on performing multiple steps of EKI with decreasing values of the regularization parameter. This strategy has a lot of similarities to the empirical Bayesian approach, where we assume a Gaussian prior but treat the regularization parameter as unknown and try to estimate it from the data (see e.g. \cite{VidPer18}). Our analysis shows that, by coupling the sample size to the regularization parameter, it becomes possible to obtain the optimal convergence rates in the zero-noise limit. This is also the major difference of the presented scheme to other versions of EKI.

\subsubsection*{Relation to multiscale methods}

The ideas behind adaptive EKI are similar to sequential multiscale methods, where one iteratively moves from a low-dimensional coarse-scale subspace to finer scales. A related work along these lines is \cite{NadPotRho18}, which also applies to ensemble methods, but considers the setting where in each step an approximate solution on a different subspace is computed. Under certain conditions on the multiscale decomposition, this approach can be shown to be equivalent to Tikhonov regularization in the full space. In contrast, the idea behind adaptive EKI is only to approximate Tikhonov regularization, in a way that achieves the same convergence order in the zero-noise limit.

\subsubsection*{Localization}

In some practical applications (e.g. numerical weather prediction \cite{HouZha16}) it is only feasible to work with ensemble sizes that are orders of magnitude smaller than the parameter dimension. In these situations, localization \cite{GreKalMiyIdeHun11} is often used to increase the effective ensemble size through incorporation of domain knowledge on the correlation structure of the parameter or observation of interest. Since adaptive EKI can be formulated both in square-root and covariance form (see \autoref{rem:covariance_form}), it can be combined with most of the existing localization methods, such as covariance localization \cite{HouMit01} or local analysis \cite{OttHunSzuZimKos04}. Note that localization for stochastic EKI has been studied in \cite{TonMor22_report}.

\section{Numerical experiments}\label{sec:numerics}

We performed numerical experiments to evaluate the performance of adaptive EKI.

\subsection*{Test problem}

We have chosen inversion of the Radon transform $L$ (see for instance \cite{Kuc13}) as our test example. The analytical results show that the large ensemble limit approximates the Tikhonov regularized solution, which we aim to verify numerically. And we also compare the different variants of EKI in terms of efficiency. As a test object, we use the classic Shepp-Logan phantom \cite{SheLog74} with size $d \times d$, $d=100$, (see \autoref{fig:shepp_logan}). This corresponds to a parameter dimension of $n := \dim \XX = d^2 = 10000$ and a measurement dimension of $m := \dim \YY = 14200$. 

\begin{figure}[H]
\begin{center}
\includegraphics[width=0.25\textwidth]{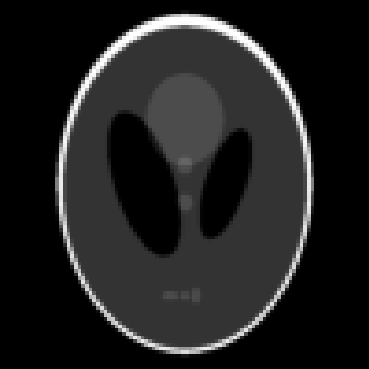}%
\captionof{figure}{The Shepp-Logan phantom}\label{fig:shepp_logan}%
\end{center}
\end{figure}

\subsection*{Data simulation}
We generated noise $\xi_s \sim \normal(0, \Id_m)$ from a standard normal distribution and then rescaled the noise by setting
\begin{equation*}
\xi := \frac{\norm{y}}{10 \norm{\xi_s}} \xi_s,
\end{equation*}
thereby ensuring a signal-to-noise ratio of exactly 10. We then used $\hat y = y + \xi$ as noisy measurement for the tested methods. We also rescaled the measurement and the observation operator by $\norm{\xi}$ so that $\delta = \norm{\obs - y} = 1$. 

\subsection*{Considered methods}
In our experiment, we set $\noicov = \Idmat_m$, and chose $\cov \in \R^{n \times n}$ equal to a discretized covariance operator of an Ornstein-Uhlenbeck process,
\begin{align*}
(\cov)_{ij} & := e^{-\norm{q_i - q_j}/h^2}, \\
\text{where}\quad q_i & := \begin{pmatrix}
\frac{i ~\mathrm{mod}~ d}{d-1}\\ \frac{\lfloor i / d \rfloor}{d-1}
\end{pmatrix} \in [0,1] \times [0,1],
\end{align*}
with correlation length $h > 0$ (we used the value $h = 0.01$). Such operators are often used as prior covariance for Bayesian MAP estimation in tomography, for example in \cite{Tar19}.
They correspond to the assumption that the correlation between individual pixels decreases exponentially with distance, where $q_i$ denotes the normalized position of the $i$-th pixel if the image is scaled to $[0,1]\times[0,1]$.
We compared the 3 different instances of adaptive EKI discussed in \autoref{sec:rates}:
\begin{itemize}
\item Standard EKI with $\alpha_k = b^k \alpha_0$, $b = \sqrt[4]{0.8}$, $J_k = \lceil b^{-4(k-1)} J_1 \rceil$, $\alpha_0=0.15$, and $J_1 = 50$.
\item Nyström-EKI with $\alpha_k = b^k \alpha_0$, $b = \sqrt{0.8}$, $J_k = \lceil b^{-2(k-1)} J_1 \rceil$, $\alpha_0=0.15$, and $J_1 = 50$.
\item SVD-EKI with $\alpha_k = b^k \alpha_0$, $b = \sqrt{0.8}$, $J_k = \lceil b^{-2(k-1)} J_1 \rceil$,$\alpha_0=0.15$, and $J_1 = 50$.
\end{itemize}
The different values of $b$ are used in order to ensure that the sequence of sample sizes $(J_k)_{k=1}^\infty$ is equal for all three methods.
Moreover, all methods used the discrepancy principle (see \autoref{de:discrepancy}) with $\tau = 1.2$. In any case, the iterations where aborted once $J_k$ was larger than $n$, since at this point the computational complexity of EKI is higher than of Tikhonov regularization.

\subsection*{Implementation}
The algorithms were implemented in Python and use efficient Numpy \cite{HarMilWalGomVir20} and SciPy \cite{VirGomOliHabRed20} routines. We used the existing implementation of the Radon transform in the scikit-image library \cite{WalSchoNunBouWar14}, and took advantage of the Ray framework \cite{MorNisWanTumLia18} to parallelize the operator evaluations. The computations were performed on a Dell XPS-15-7590 Laptop with 12 2.60 GHz CPUs and 15.3 GiB RAM.

\subsection*{Convergence of adaptive EKI}
For each iteration, we evaluated the relative reconstruction error
\begin{align*}
e_\mathrm{rel}(x) := \frac{\norm{x - \xtrue}}{\norm{\xtrue}}.
\end{align*}
The results are visualized in \autoref{fig:speedplot}. Note that every iteration is computationally more expensive than the previous one since the sample size $J_k$ increases steadily. While the Nyström-EKI and the SVD-EKI methods were able to satisfy the discrepancy principle after 18 and 16 iterations, respectively (with sample size $J_{18} = 2284$ and $J_{16}=1461$), the Standard EKI-iteration was not able to satisfy the discrepancy principle for a sample size less than $n$. Apart from that, one clearly sees that Nyström-EKI and SVD-EKI both significantly outperform Standard-EKI. Consistent with \autoref{schmidt}, one may observe that SVD-EKI yields the most accurate reconstruction for given sample size.

\begin{center}
\includegraphics[width=0.6\textwidth]{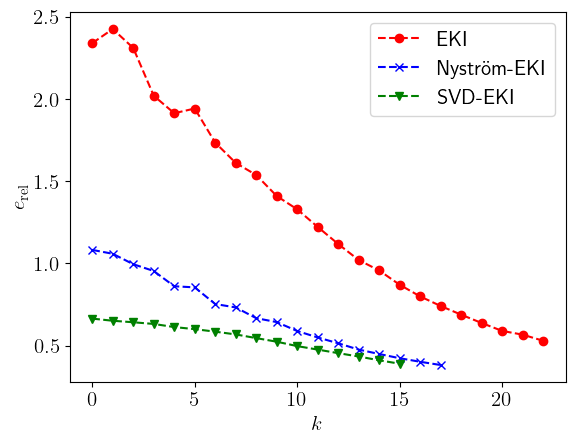}%
\captionof{figure}{The adaptive EKI, Nyström-EKI and SVD-EKI iterations. The $x$-axis denotes the iteration number. The $y$-axis denotes the relative reconstruction error $e_\mathrm{rel}$. The Standard-EKI iteration was not able to satisfy the discrepancy principle for $J_k < n$.}\label{fig:speedplot}%
\end{center}

\subsection*{Comparison of Standard-EKI with Nyström-EKI}
In \autoref{fig:standard_vs_nys}, we visually compare the reconstruction with Standard-EKI to the reconstruction with Nyström-EKI. Both reconstructions use the same value of $\alpha$ and sample size $J=2000$. One can see that the standard method is much more noisy than the Nyström method. This noise does not come from the noisy measurement, it is introduced by the sampling process.

\begin{figure}[!h]
\begin{center}
\begin{subfigure}[t]{0.2\textwidth}
	\centering
		\includegraphics[height=0.8\textwidth]{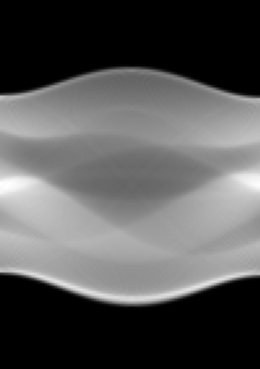}
		\caption{\footnotesize The measurement $y$.}
	\end{subfigure}\hspace{0.05\textwidth}%
	\begin{subfigure}[t]{0.2\textwidth}
	\centering
		\includegraphics[height=0.8\textwidth]{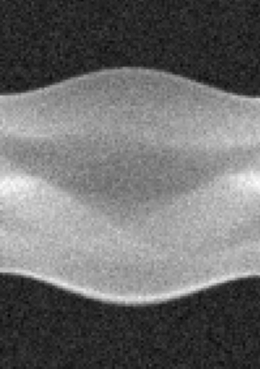}
		\caption{\footnotesize The noisy data $\hat y$.}
	\end{subfigure}\hspace{0.05\textwidth}%
	\begin{subfigure}[t]{0.2\textwidth}
	\centering
		\includegraphics[height=0.8\textwidth]{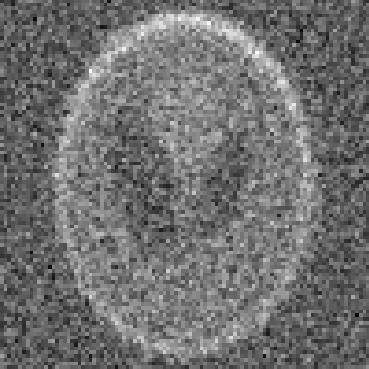}
		\caption{\footnotesize\centering Reconstruction from noisy data with Standard-EKI.}
	\end{subfigure}\hspace{0.05\textwidth}%
	\begin{subfigure}[t]{0.2\textwidth}
	\centering
		\includegraphics[height=0.8\textwidth]{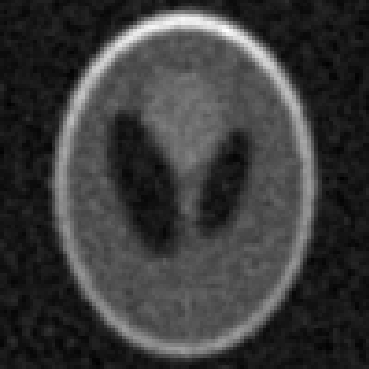}
		\caption{\footnotesize\centering Reconstruction from noisy data with Nyström-EKI.}
	\end{subfigure}
	\caption{Reconstruction of the Shepp-Logan phantom from noisy data using EKI with sample size $J=2000$, corresponding to $1/5$ of the parameter dimension.}\label{fig:standard_vs_nys}
\end{center}
\end{figure}

\subsection*{Convergence to Tikhonov regularization for large sample sizes}

In \autoref{fig:nys_to_tik}, we have plotted the reconstruction with Nyström-EKI for increasing values of $J$. For $J = 500$, the reconstruction is hardly useful. However, for $J=2000$ the reconstruction is already almost comparable to the Tikhonov reconstruction, although a little bit blurred. For higher values of $J$, the improvement is only marginal. This shows that the presence of noise allows considerable a-priori (that is, not using knowledge on $L$ or $\obs$) dimensionality reduction.

\begin{figure}[!h]
\begin{center}
	\begin{subfigure}[t]{0.2\textwidth}
	\centering
		\includegraphics[width=0.8\textwidth]{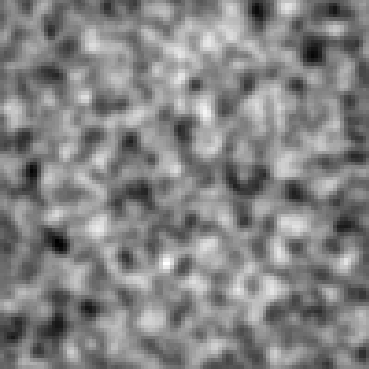}
		\caption{\footnotesize $J=100$.}
	\end{subfigure}%
	\begin{subfigure}[t]{0.2\textwidth}
	\centering
		\includegraphics[width=0.8\textwidth]{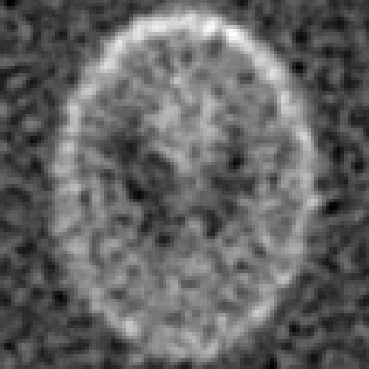}
		\caption{\footnotesize $J=500$.}
	\end{subfigure}%
	\begin{subfigure}[t]{0.2\textwidth}
	\centering
		\includegraphics[width=0.8\textwidth]{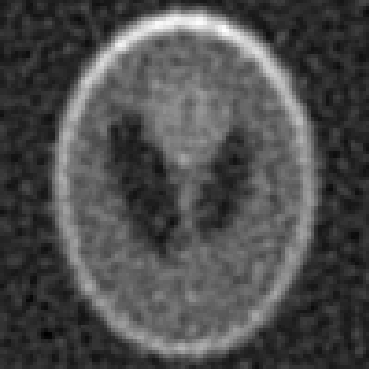}
		\caption{\footnotesize $J=1000$.}
	\end{subfigure}%
	\begin{subfigure}[t]{0.2\textwidth}
	\centering
		\includegraphics[width=0.8\textwidth]{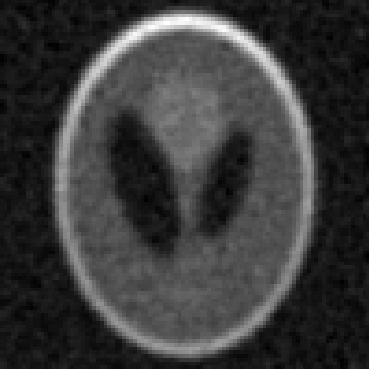}
		\caption{\footnotesize $J=2000$.}
	\end{subfigure}
	
	\begin{subfigure}[t]{0.2\textwidth}
	\centering
		\includegraphics[width=0.8\textwidth]{figure4_nys_2000.png}
		\caption{\footnotesize $J=3000$.}
	\end{subfigure}%
	\begin{subfigure}[t]{0.2\textwidth}
	\centering
		\includegraphics[width=0.8\textwidth]{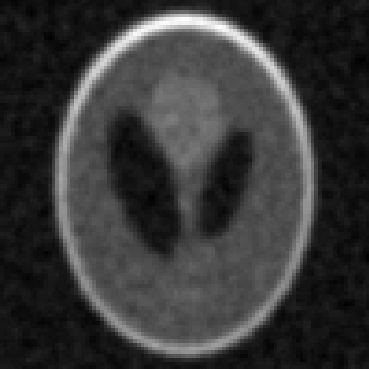}
		\caption{\footnotesize $J=5000$.}
	\end{subfigure}%
	\begin{subfigure}[t]{0.2\textwidth}
	\centering
		\includegraphics[width=0.8\textwidth]{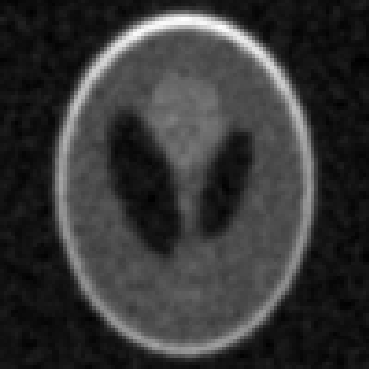}
		\caption{\footnotesize $J=8000$.}
	\end{subfigure}%
	\begin{subfigure}[t]{0.2\textwidth}
	\centering
		\includegraphics[width=0.8\textwidth]{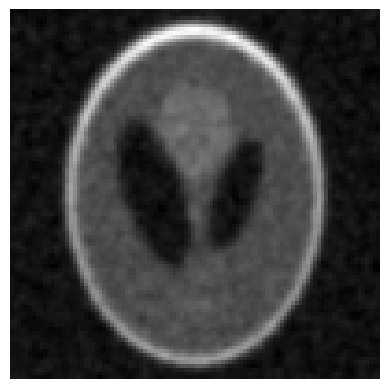}
		\caption{\footnotesize Tikhonov regularization.}
	\end{subfigure}%
\end{center}
\caption{Reconstruction with Nyström-EKI for different sample sizes $J$, using noisy data with a signal-to-noise ratio of $10$.}\label{fig:nys_to_tik}
\end{figure}

We also repeated the experiment for fixed regularization parameter $\alpha=0.03$ and different values of $J$ in order to examine the convergence estimate from \autoref{subsec:tikhonov} numerically. In \autoref{fig:convergence}, we plotted the approximation error with respect to Tikhonov regularization, normalized with $\norm{\xtrue}$, i.e.
\begin{align*}
e_\mathrm{app}(x; \alpha) := \frac{\norm{x - \xtika}}{\norm{\xtrue}}.
\end{align*}
In accordance with \autoref{large_ensemble_convergence}, the approximation error of Standard-EKI decreases like $J^{-1/2}$. However, it is still significant even if the number of ensembles is close to $n$. With Nyström-EKI or SVD-EKI, the approximation error becomes negligible even for relatively small sample sizes.

\begin{center}
\includegraphics[width=0.6\textwidth]{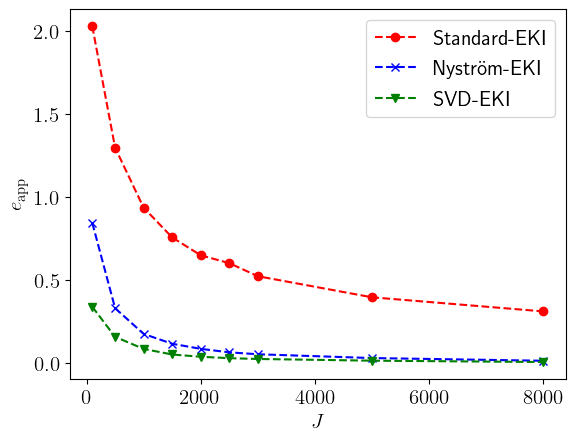}%
\captionof{figure}{The Standard-EKI, Nyström-EKI and SVD-EKI iterations for fixed regularization paramter $\alpha$ and varying sample size. The $x$-axis denotes the sample size $J$. The $y$-axis denotes the relative approximation error $e_\mathrm{app}$.}\label{fig:convergence}%
\end{center}

\subsection*{Divergence for small values of $\alpha$}
Keeping the sample size fixed at $J=2000$, we then repeated the experiment for different values of $\alpha$ (see \autoref{fig:divergence}). One sees that the approximation error of all three methods explodes as $\alpha \to 0$, which demonstrates the necessity of adapting the sample size. Again, Nyström-EKI and SVD-EKI are superior to Standard-EKI.

\begin{center}
\includegraphics[width=0.6\textwidth]{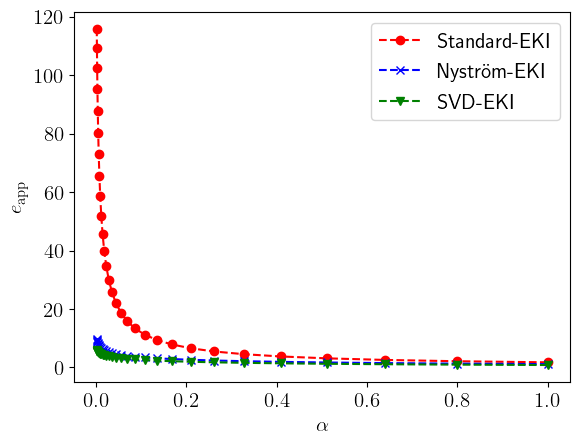}%
\captionof{figure}{The Standard-EKI, Nyström-EKI and SVD-EKI iterations for fixed sample size $J$ and varying regularization parameter. The $x$-axis denotes the regularization parameter $\alpha$. The $y$-axis denotes the scaled approximation error $ e_\mathrm{app}$ for EKI with sample size $J=2000$. As $\alpha$ approaches 0, the approximation error explodes.}\label{fig:divergence}%
\end{center}

\section{Conclusions}\label{sec:conclusions}

We have shown that ensemble Kalman inversion is a convergent regularization method if the sample size is adapted to the regularization parameter. The interpretation of EKI as a low-rank aproximation of Tikhonov regularization shows that it provides a trade-off between exactness and computational cost by shrinking the search space in which we try to reconstruct the unknown parameter $x$. This approach is suited for problems where the adjoint is not available and the noise is significant, since then the optimal regularization parameter $\alpha$ will typically be larger and a good approximation to the Tikhonov-regularized solution can be achieved for relatively small sample sizes (see \autoref{fig:divergence}).

It is important to note that the dimensionality reduction in EKI is completely a-priori. It uses no knowledge about the forward operator $L$ or the measurement $\obs$. This has the advantage that it also works in the case where the adjoint of $L$ is not available. On the other hand, if one has access to the adjoint of $L$, one can compute a low-rank approximation of the whole operator $\cov^{1/2} L^* R^{-1} L \cov^{1/2}$ instead \cite{FlaWilAkcHilWaa11}. This can yield superior results as it allows to also exploit the spectral decay of the forward operator $L$ \cite{SpaSolCuiMarTen15}.

While EKI was originally developed for nonlinear inverse problems, our insights from the linear case -- in particular the need for adapting the sample size to the noise level -- can serve as an Ansatz for an analysis of EKI as a regularization method for nonlinear inverse problems.

After all, the basic ideas of ensemble methods are simple and constitute a very general way to obtain linear dimensionality reduction and algorithms for black-box inverse problems. Therefore, another natural direction of research is to study the resulting stochastic approximations of classical iterative regularization methods, such as the iteratively regularized Gauss-Newton iteration \cite{Bak92}, and compare their performance to EKI for the case of nonlinear inverse problems.

\appendix
\subsection*{Acknowledgements}FP and OS are supported by the Austrian
Science Fund (FWF) with project I3661-N27 (Novel Error Measures and Source Conditions of
Regularization Methods for Inverse Problems). Moreover, FP and OS are supported by the Austrian
Science Fund (FWF), with SFB F68, project F6807-N36 (Tomography with
Uncertainties). The financial support by the Austrian Federal Ministry for Digital and Economic
Affairs, the National Foundation for Research, Technology and Development and the Christian Doppler
Research Association is gratefully acknowledged.

\section*{References}
\renewcommand{\i}{\ii}
\printbibliography[heading=none]

\section{Appendix: Random elements of Hilbert spaces} \label{sec:appendix}

We recapitulate basic notions from probability theory on Hilbert spaces.

\begin{definition}[Random element, expectation, covariance]
Let $(\Omega,\mathcal{F},\PP)$ denote a probability space. 
\begin{enumerate}
\item 
A \emph{random element} of a real Hilbert space $\XX$ 
is a measurable function $X: \Omega \to \XX$. 
We call 
\begin{equation} \label{eq:E}
\mathcal{M}(\XX) := \set{X:\Omega \to \XX: X \text{ is measurable}}.
\end{equation}
\item A \emph{random continuous linear operator} from $\XX$ to $\YY$ is a measurable map $A: \Omega \to \BL(\XX;\YY)$.
\item The expectation of a random element $X$ of $\mathcal{M}(\XX)$ is defined as
\begin{align*}
\Exp X = \int_\Omega X(\omega) \d \PP(\omega) \in \XX
\end{align*}
\item Furthermore, its covariance operator $\Cov{X}: \XX \to \XX$ is defined by
\begin{align*}
\Cov{X}u = \int_\Omega \xinner{X(\omega)-\Exp X}{u}(X(\omega) - \Exp X) \d \PP(\omega), \qquad u \in \XX.
\end{align*}
\item 
We call a random element $X:\Omega \to \XX$ of a Hilbert space $\XX$ \emph{Gaussian} if for every continuous linear functional $L \in \XX^*$, $LX: \Omega \to \R$ is a Gaussian random element of $\R$. That is there exist $\sigma_L > 0$ and $m_L \in \R$ such that for all $z \in \R$
\begin{equation} \label{eq:Gaussian}
\PP\left(\set{\omega:LX(\omega) \leq z} \right) = \frac{1}{\sqrt{2\pi \sigma_L^2}} 
\int_{-\infty}^z \e^{-\frac{(\xi-m_L)^2}{2 \sigma_L^2}} d\xi.
\end{equation} 
\item It can be shown that for every $m \in \XX$ and every positive and self-adjoint trace class operator $C$ there exists a unique Gaussian random element $X$ with $\Exp X= m$ and $\Cov X = C$. 
In that case, we will use the notation $X \sim \normal(m, C)$.
\item Let $\bm X = (X_1,\ldots,X_J) \in \mathcal(\XX)^J$ be a random ensemble. We call the mapping
\begin{equation}
\begin{aligned}
\Encov(\xens): \mathcal \HH &\to \HH, \\
v & \mapsto \frac{1}{J} \sum_{j=1}^J (X_j - {\bf{\bar X}} ) \inner{X_j - {\bf{\bar X}}}{v}
\end{aligned}\label{eq:encov}
\end{equation}
the \emph{sample covariance}.
\end{enumerate}
\end{definition}

Furthermore, we recall Markov's inequality as it is used in the proof of \autoref{complement_estimate} (see e.g. \cite[lemma 4.1]{Kal02}).

\begin{lemma}[Markov]\label{markov}
Let $X: \Omega \to [0,\infty)$ be a nonnegative real-valued random variable, $p \in [1, \infty)$ and $a > 0$. Then
\begin{align*}
\PP(X > a) \leq \frac{\Exp{X^p}}{a^p}.\end{align*}
\end{lemma}

\section{EKI with stochastic perturbations}\label{sec:eki_with_perturbations}

The deterministic formulation of EKI that we considered in this paper (see \autoref{de:EKI}) is based on the ensemble-transform Kalman filter (ETKF) by Bishop, Etherton and Majumdar \cite{BisEthMaj01}. It was for example also studied in \cite{ChaTon21}. In contrast, the original formulation of EKI \cite{IglLawStu13} was based on the EnKF with perturbation of measurements \cite{BurLeeEve98}. While the ETKF updates the current state estimate $\Xk$ and the ensemble anomaly $\Ank$ directly, the EnKF iterates a complete ensemble $\xens_k$ and updates each ensemble member individually. We call this variant the stochastic form of EKI:
\begin{definition}[Stochastic EKI]\label{de:EKI_perturbed}
Given is $R \in \BL(\YY;\YY)$ and an ensemble \\
$\xens_{0} = (X_{0,1},\ldots, X_{0,J})$ of independent and identically distributed random elements $X_{0,1},\ldots,X_{0,J}$.
\begin{description}
\item \emph{Initialization:}  Set $\bm C_0 =\Encov(\xens_0)$ (see \autoref{eq:encov}).
\item \emph{Iteration ($k \to k+1$):} Let $\xi_{k,1},\ldots,\xi_{k,J}$ be independent and identically distributed Gaussian random elements of $\YY$ with $\xi_{k,1},\ldots,\xi_{k,J} \sim \normal(0,\noicov)$. For each $j \in \lbrace 1,\ldots,J \rbrace$, set
\begin{equation}\label{eq:EKI_perturbed}
\Xj_{k+1,j} = \Xj_{k,j} + \Ck L^* \left(L \Ck L^* + R \right)^{-1} (\obs + \xi_{k,j} - L \Xj_{k,j}),
\end{equation}
and then set $\Cj_{k+1} = \Encov(\xens_{k+1})$.
\end{description}
\end{definition}
In the present paper, we have focused on the deterministic version of EKI given by \autoref{de:EKI}, since it has been observed to perform more reliably in practice as it does not introduce additional noise at every step of the iteration \cite{ChaTon21}. Nevertheless, both variants seem to be equivalent in the large ensemble-limit. In the linear case, this has been proven:

\begin{proposition}
Suppose that \autoref{assumption} holds and $\cov$ is in the trace class. Let $(\Xk)_{k=1}^\infty$ be the deterministic EKI iteration (see \autoref{de:EKI}), and $(\xens_k)_{k=1}^\infty$ be the stochastic EKI iteration (see \autoref{de:EKI_perturbed}). Let $p \in [1,\infty)$ and $k \in \N$. Then, we have
\begin{align*}
& \Lpxnorm{\Xk - \xensmean_k}{p} \to 0 \\
\text{and} \quad & \Lpxopnorm{\Ank \Ank^* - \Encov(\xens_k)}{p} \to 0,
\end{align*}
as $J \to \infty$.
\end{proposition}
\begin{proof}
Compare \cite[theorem 5.2]{GlaMonTra09_report} and \cite[theorem 6.1]{KwiMan15}.
\end{proof}

While we do not know of a corresponding proof in the nonlinear case, it has been observed in numerical experiments that also in that case both the deterministic form and the stochastic of the ensemble Kalman filter converge to the same limit \cite{RaaStoEve19}.

\end{document}